\documentclass[a4paper,12pt,reqno]{amsart}
\usepackage{amsmath,amsfonts,amssymb,amsthm,mathtools,autobreak,atendofenv}
\usepackage{bm,ascmac,array,framed,physics,empheq}
\usepackage{circledsteps}
\usepackage[T1]{fontenc}
\usepackage{textcomp}
\usepackage[normalem]{ulem}
\usepackage{mathrsfs,enumerate}
\usepackage[shortlabels]{enumitem}
\usepackage[all]{xy}
\usepackage[dvipdfmx]{graphicx}
\usepackage{tikz}
\usepackage{url}
\usetikzlibrary{cd,intersections,calc,arrows,matrix,shapes.geometric,positioning}
\usepackage{tcolorbox}
\usepackage{color}
\usepackage{appendix}
\usepackage{stackengine}
\usepackage{here}
\usepackage[dvipdfmx,hidelinks]{hyperref}
\hypersetup{unicode,bookmarksnumbered=true,hidelinks,final}
\theoremstyle{definition}
\newtheorem{theorem}{Theorem}[section]
\newtheorem{definition}[theorem]{Definition}
\newtheorem{proposition}[theorem]{Proposition}
\newtheorem{lemma}[theorem]{Lemma}
\newtheorem{corollary}[theorem]{Corollary}
\newtheorem{remark}[theorem]{Remark}
\newtheorem{assumption}[theorem]{Assumption}
\newtheorem*{proposition*}{Proposition}
\newcommand{\coeff}{\mathop{\mathrm{coeff}}}
\mathtoolsset{showonlyrefs=true}
\numberwithin{equation}{section}
\makeatletter
\allowdisplaybreaks
\makeatother
\begin{document}
\title[Moduli of Higgs pairs]{The moduli space of Higgs pairs}
\author{Jun Sasaki}
\date{}
\address{Department of Mathematics, Institute of Science Tokyo, 2-12-1, O-okayama, Meguro, 152-8551, Japan}
\email{sasaki.j.ac[@]m.titech.ac.jp, sasaki.j.fb69[@]m.isct.ac.jp}
\begin{abstract}
  In this paper, we study the moduli space of Higgs pairs, which can be considered as a generalization of holomorphic pairs. Higgs pairs are an example of quiver bundles. We introduce the notion of $\tau$-stability of Higgs pairs for $\tau\in\mathbb{R}$ and establish the Kobayashi-Hitchin correspondence for Higgs pairs. The differential-geometric objects corresponding to stable Higgs pairs is called the vortex equations for Higgs bundles. We analyze the moduli space of stable Higgs pairs when the base space of vector bundle is a compact Riemann surface and obtaine the following results.
  \begin{itemize}
    \item We prove that the moduli space is non-singular complex manifold for a suitable choice of $\tau$.
    \item We determine the Poincar\'{e} polynomial of the moduli space for $\rank 2$ bundle.
    \item We construct a map from the moduli space of stable Higgs pairs to the moduli space of stable Higgs bundles and proved that the map is a fibration under suitable assumptions.
  \end{itemize}
\end{abstract}
\maketitle
\setcounter{section}{-1}
\section{Introduction}
Let $\left(M,g\right)$ be an $n$-dimensional compact K\"{a}hler manifold and $\omega$ be the K\"{a}hler form of $\left(M,g\right)$. A pair $\left(E, D^{\prime\prime}\right)$ is called a Higgs bundle over $M$ if $E$ is a smooth complex vector bundle over $M$ and $D^{\prime\prime}:\Omega^0\left(E\right)\to\Omega^1\left(E\right)$ is a $\mathbb{C}$-linear mapping satisfying the Leibniz rule
\[D^{\prime\prime}\left(fs\right)=fD^{\prime\prime}s+s\otimes \bar{\partial}f, \quad \forall{f}\in C^\infty\left(M\right),\forall{s}\in\Omega^0\left(E\right),\]
and the integrability condition
\[D^{\prime\prime}\circ D^{\prime\prime}=0:\Omega^0\left(E\right)\to\Omega^2\left(E\right),\]
where $\Omega^k\left(E\right)$ denotes a space of the $E$-valued smooth $k$-forms for $k\in\mathbb{Z}_{\geq 0}$. A Higgs bundle is considered as a generalization of holomorphic vector bundle since if we write $D^{\prime}=\bar{\partial}^E+\theta$ according to the decomposition $\Omega^1\left(E\right)=\Omega^{1,0}\left(E\right)\oplus\Omega^{0,1}\left(E\right)$, $\bar{\partial}^E$ defines the holomorphic structure of $E$.

Let $\left(E,h\right)$ be a smooth Hermitian vector bundle over $M$ and $\mathcal{H}\left(E,h\right)$ denotes the space of integrable unitary connections, that is, $\mathcal{H}\left(E,h\right)$ consists of a unitary connection $\nabla$ such that the $\left(0,2\right)$-part of curvature $R\left(\nabla\right)$ of $\nabla$ vanishes. A connection $\nabla\in\mathcal{H}\left(E,h\right)$ is called a Hermitian-Einstein (HE) connection of $\left(E,h\right)$ if it satisfies the following equation.
\[\sqrt{-1}\Lambda R\left(\nabla\right)=c\,\textrm{id}_E,\quad c=\frac{2\pi\mu\left(E\right)}{\textrm{Vol}\left(M,g\right)},\]
where $\Lambda$ is the contradiction with the K\"{a}hler form $\omega$. It is known that HE connections attain the minimum of the functional $\textrm{YM}:\mathcal{H}\left(E,h\right)\to\mathbb{R}_{\geq 0}$ defined by, for $\nabla\in\mathcal{H}\left(E,h\right)$, $\textrm{YM}\left(\nabla\right)=\left\|R\left(\nabla\right)\right\|_{L^2\left(M\right)}^2$.
 Uhlenbeck and Yau \cite{Uhlenbeckyau} proved a theorem called the Kobayashi-Hitchin (KH) correspondence, which asserts that a holomorphic vector bundle is stable if and only if there is a Hermitian metric such that its Chern connection are HE connection and the metric is unique up to constant multiplicity. This theorem expresses an energy-minimizing condition in terms of an algebraic condition.

The KH correspondence is valid for Higgs bundles. Let $\left(E,h\right)$ be a smooth Hermitian vector bundle over $M$. For a connection $D$ on $E$, consider a decomposition $D=D^\prime+D^{\prime\prime}$ of $D$ as in \cite[p. $13$]{Simpson2} and consider $\mathscr{H}\left(E,h\right)$ consisting of connections $D$ on $E$ such that $D^{\prime\prime}$ determines a structure of Higgs bundle in $E$. A connection $D\in\mathscr{H}\left(E,h\right)$ is called a Hermitian-Yang-Mills (HYM) connection of $\left(E,h\right)$ if it satisfies the following equation.
\[\sqrt{-1}\Lambda R\left(D\right)=c\,\textrm{id}_E,\quad c=\frac{2\pi\mu\left(E\right)}{\textrm{Vol}\left(M,g\right)}.\]
It is also known that HYM connections attain the minimum of the functional $\textrm{HYM}:\mathscr{H}\left(E,h\right)\to\mathbb{R}_{\geq 0}$ defined by, for $D\in\mathscr{H}\left(E,h\right)$, $\textrm{HYM}\left(D\right)=\left\|R\left(D\right)\right\|_{L^2\left(M\right)}^2$. The HYM condition for $D\in\mathscr{H}\left(E,h\right)$ can be expressed equivalently by the following equation.

Hitchin and Simpson \cite{Hitchin,Simpson} proved that a Higgs bundle are stable if and only if there is a Hermitian metric such that its Hitchin-Simpson (HS) connection, which is a generalization of the Chern connection for Higgs bundle, are HYM connection and the metric is unique up to constant multiplicity.

For smooth Heritian vector bundles $\left(E_1,h_1\right),\left(E_2,h_2\right)$ over $M$, Garc\'{i}a-Prada \cite{Garcia-Prada} considered an $SU\left(2\right)$-equivalent vector bundle 
\[\left(F=p^*E_1\oplus p^*E_2\otimes q^*O\left(2\right), h=p^*h_1\oplus p^*h_2\otimes q^*h^\prime\right)\]
over $M\times\mathbb{CP}^1$, where $p$ and $q$ are projection from $M\times\mathbb{CP}^1$ to $M$ and $\mathbb{CP}^1$, respectively and $h^\prime$ is the standard Hermitian metric on $O\left(2\right)$. For a real number $\tau$, we define a K\"{a}hler form $\omega_\sigma$ on $M\times\mathbb{CP}^1$ defined by
\[\omega_\sigma=p^*\omega\oplus q^*\sigma\omega_{FS},\quad \sigma=\frac{2\rank E_2}{\displaystyle\frac{\rank E_1+\rank E_2}{4\pi}\tau-\frac{\deg E_1+\deg E_2}{\textrm{Vol}\left(M,g\right)}},\]
where $\omega_{FS}$ is the Fubini-Study metric on $\mathbb{CP}^1$ with $\textrm{Vol}\left(\mathbb{CP}^1,\omega_{FS}\right)=1$. By using dimension reduction, Garc\'{i}a-Prada established a one-to-one correspondence between $SU\left(2\right)$-invariant HE connections on $\left(F,h\right)$ with respect to $\omega_\sigma$ and solutions to the coupled $\tau$-vortex equations, which are given by, for $\nabla^{E_1}\in\mathcal{H}\left(E_1,h_1\right),\nabla^{E_2}\in\mathcal{H}\left(E_2,h_2\right)$ and $\phi\in\Omega^0\bigl(\textrm{Hom}\left(E_2,E_1\right)\bigr)$,
\begin{empheq}[left=\empheqlbrace]{align}
  &\sqrt{-1}\Lambda R\left(\nabla^{E_1}\right)+\frac{1}{2}\phi\circ\phi^*=\frac{\tau}{2}\textrm{id}_{E_1},\\
  &\sqrt{-1}\Lambda R\left(\nabla^{E_2}\right)-\frac{1}{2}\phi^*\circ\phi=\frac{\tau^\prime}{2}\textrm{id}_{E_2},\\
  &\left(\nabla^{\textrm{Hom}\left(E_2,E_1\right)}\right)^{0,1}\phi=0,
\end{empheq}
where $\tau,\tau^\prime\in\mathbb{R}$ satisfy the following relation.
\[\tau\rank E_1+\tau^\prime\rank E_2=\frac{4\pi}{\textrm{Vol}\left(M,g\right)}\left(\deg E_1+\deg E_2\right).\]

In $2025$, Ono \cite{Ono} developed a dimensional reduction of the HYM equation in a manner analogous to Garc\'{i}a-Prada's argument for the HE equation and thereby he introduced the doubly-coupled $\tau$-vortex equations, which are given by, for $D_{E_1}\in\mathscr{H}\left(E_1,h_1\right), D_{E_2}\in\mathscr{H}\left(E_2,h_2\right), \phi\in\Omega^0\bigl(\textrm{Hom}\left(E_2,E_1\right)\bigr)$ and $\psi\in\Omega^0\bigl(\textrm{Hom}\left(E_1,E_2\right)\bigr)$,
\begin{empheq}[left=\empheqlbrace]{align}
  &\sqrt{-1}\Lambda R\left(D_{E_1}\right)+\frac{1}{2}\phi\circ\phi^*-\frac{1}{2}\psi^*\circ\psi=\frac{\tau}{2}\textrm{id}_{E_1},\\
  &\sqrt{-1}\Lambda R\left(D_{E_2}\right)-\frac{1}{2}\phi^*\circ\phi+\frac{1}{2}\psi\circ\psi^*=\frac{\tau^\prime}{2}\textrm{id}_{E_2},\\
  &D_{\textrm{Hom}\left(E_2,E_1\right)}^{\prime\prime}\phi=0,\quad D_{\textrm{Hom}\left(E_1,E_2\right)}^{\prime\prime}\psi=0, \label{eq:third equation}\\
  &\phi\circ\psi=0,\quad \psi\circ\phi=0 \label{eq:last equation}.
\end{empheq}
Ono also introduced a notion of Higgs quadruplet and its stability. a Higgs quadruplet is a quadruplet $\left(\left(E_1,D_{E_1}^{\prime\prime}\right),\left(E_2,D_{E_2}^{\prime\prime}\right),\phi,\psi\right)$ consisting of Higgs bundles $\left(E_1,D_{E_1}^{\prime\prime}\right)$ and $\left(E_2,D_{E_2}^{\prime\prime}\right)$ over $M$ and $\phi\in\Omega^0\bigl(\textrm{Hom}\left(E_2,E_1\right)\bigr),\psi\in\Omega^0\bigl(\textrm{Hom}\left(E_1,E_2\right)\bigr)$ such that they satisfy \eqref{eq:third equation} and \eqref{eq:last equation}. Ono established the KH correspondence for Higgs quadruplets, that is, a Higgs quadruplet $\left(\left(E_1,D_{E_1}^{\prime\prime}\right),\left(E_2,D_{E_2}^{\prime\prime}\right),\phi,\psi\right)$ is $\tau$-stable if and only if there are Hermitian metrics on $E_1$ and $E_2$ such that their HS connections, $\phi$ and $\psi$ are solution to the doubly-coupled $\tau$-vortex equations and the metrics are unique up to constant multiplicity.

In this paper, we consider the doubly-coupled $\tau$-vortex equation when $E_2$ is the trivial line bundle. We set $E=E_1$ and $h=h_1$. The third and last equations imply that $\phi\in\Omega^0\left(E\right)$ and $\psi\in\Omega^0\left(E^*\right)$ are holomorphic and $\textrm{supp}\,\phi$ and $\textrm{supp}\,\psi$ are disjoint. Thus at least one of $\phi$ and $\psi$ vanishes. Suppose that $\psi$ vanishes and we set $s=\phi$. Then as a special case of the doubly-coupled $\tau$-vortex equations, we introduce the $\tau$-vortex equations for Higgs bundles, which is given by, for $D\in\mathscr{H}\left(E,h\right)$ and $s\in\Omega^0\left(E\right)$,
\[\sqrt{-1}\Lambda R\left(D\right)+\frac{1}{2}s\circ s^*=\frac{\tau}{2}\textrm{id}_{E_1},\quad D^{\prime\prime}s=0.\]
These equations can be considered as a generalization of $\tau$-vortex equations introduced by Bradlow \cite{Bradlow2}.

We also introduce a notion of Higgs pair, which consists of a strucutre of Higgs bundle $D^{\prime\prime}$ in $E$ and a smooth section $s\in\Omega^0\left(E\right)$ such that $D^{\prime\prime}s=0$ holds. We can consider Higgs pairs as a generalization of holomorphic pair. \'{A}lvarez-C\'{o}nsul and Garc\'{i}a-Prada \cite{Alvarez-Consul-Garcia-Prada} introduced quiver bundles over compact K\"{a}hler manifolds and established the KH correspondence for quiver bundles. Holomorphic pairs, Higgs bundles, Higgs quadruplets and Higgs pairs are all examples of quiver bundles. Hence we obtain the KH correspondence for Higgs pairs. A Higgs pair $\left(D^{\prime\prime},s\right)$ is said to be $\tau$-stable if it satisfies the following conditions.
\begin{enumerate}
    \item $\displaystyle\mu\left(F\right)<\frac{\textrm{Vol}\left(M,g\right)}{4\pi}\tau$ holds for every Higgs subsheaf $F$ of $\left(E,D^{\prime\prime}\right)$ with $0<\rank F$.
    \item $\displaystyle\mu\left(E/F\right)>\frac{\textrm{Vol}\left(M,g\right)}{4\pi}\tau$ holds for every Higgs subsheaf $F$ of $\left(E,D^{\prime\prime}\right)$ with $s\in H^0\left(F\right)$ and $0<\rank F<\rank E$.
\end{enumerate}
The KH correspondence for Higgs pairs asserts that for a Higgs pair $\left(D^{\prime\prime},s\right)$, the following are equivalent.
\begin{enumerate}[$\left(1\right)$]
  \item There exists a Hermitian metric $h$ on $E$ such that $\left(D^{\prime\prime},s\right)$ is a solution to the $\tau$-vortex equation for Higgs bundle \eqref{eq:vortex_Higgs}.
  \item $\left(D^{\prime\prime},s\right)$ satisfies either $\left(a\right)$ or $\left(b\right)$.
  \begin{enumerate}[$\left(a\right)$]
    \item $\left(E,D^{\prime\prime}\right)$ is $\tau$-stable.
    \item There exists Higgs subbundles $E_s, E^\prime$ of $\left(E,D^{\prime\prime}\right)$ such that $E$ splits holomorphically as $E=E_s\oplus E^\prime$, where $E_s$ and $E^\prime$ satisfy the following conditions.
    \begin{itemize}
      \item $E^\prime$ is poly-stable Higgs sheaf with $\displaystyle\mu\left(E^\prime\right)=\frac{\textrm{Vol}\left(M,g\right)}{4\pi}\tau$ holds.
      \item $E_s$ contains the section $s$ and a Higgs pair $\left(E_s,s\right)$ is $\tau$-stable.
    \end{itemize}
  \end{enumerate}
\end{enumerate}
Therefore if we choose $\tau$ so that $\displaystyle\tau\ne\frac{4\pi}{\textrm{Vol}\left(M,g\right)}\mu\left(F\right)$, for any Higgs subbundle $F$ of $E$, conditions $\left(1\right)$ and $\left(2\right)\left(a\right)$ are equivalent.

The main result of this paper concerns the moduli space of $\tau$-stable Higgs pairs, denoted by $\mathscr{M}_{\textrm{HP}}^{\textrm{st}}$, when the base space of fixed vector bundle is a compact Riemann surface. Firstly, we prove that the moduli space is a non-singular complex manifold when we choose $\tau\in\mathbb{R}$ appropriately. The deformation of quiver bundles are studied by Gothen and King \cite{Gothen-King}. However, a Higgs pair is not merely a quiver bundle. We require that a Higgs field $\theta$ and a section $s$ are related by $\theta\left(s\right)=0$. When we consider the deformation of Higgs pairs, we must deform in such a way that this relation is preserved. Thus we cannot apply their theory. Therefore we descrive the $2$-nd cohomology group of deformation complex of Higgs pairs explicitly and prove that this vanishes. (See Proposition \ref{prop:H^2_vanishing_1} and \ref{prop:H^2_vanishing_2}.) Secondly, we determine the Poincar\'{e} polynomial of the space $\mathscr{M}_{\textrm{HP}}^{\textrm{st}}$ for $\rank 2$ bundle. Hitchin \cite{Hitchin} determined the Poincar\'{e} polynomial of the moduli space of stable $\rank 2$ and odd degree Higgs bundles with fixed determinant by using Morse theory. Biswas and Schumacher \cite{Biswas} proved that the moduli space of quiver bundles admits a K\"{a}hler metric. Thus the moduli space $\mathscr{M}_{\textrm{HP}}^{\textrm{st}}$ is also a K\"{a}hler manifold. Also, $\mathscr{M}_{\textrm{HP}}^{\textrm{st}}$ has a Hamiltonian circle action. By analyzing the critical set of the moment map associated with the action, we obtain the Poincar\'{e} polynomial of $\mathscr{M}_{\textrm{HP}}^{\textrm{st}}$. (See Theorem \ref{thm:Poincare}.) Lastly, we construct a map from the moduli space $\mathscr{M}_{\textrm{HP}}^{\textrm{st}}$ to the moduli space of stable Higgs bundles under a certain assumption. This result can be considered as a generalization of Bradlow and Daskalopoulos's one \cite[Proposition $6.1$]{Bradlow3}. (See Proposition \ref{prop:fibration}.)

The notion of a Higgs pair was first introduced by Mehta \cite{Mehta} as a Higgs triple. Compared with his work, he did not prove that the entire moduli space of $\tau$-stable Higgs pairs is smooth and established only the smoothness of a certain open subset. Moreover, he proved the smoothness by showing that the cup product $H^1\times H^1\to H^2$ vanishes (cf. \cite[Proposition $4.5$]{Mehta}) and did not prove that the $2$-nd cohomology group of the deformation complex vanishes. On the other hand, in this paper, we prove the smoothness of the entire of the moduli space by showing that the $2$-nd cohomology group vanishes directly. He also considered the $S^1$-action on the moduli space. However, the action is on the section $s$ in his paper, whereas it is on the Higgs field $\theta$ in this paper.

This paper is organized as follows.
In \textbf{Section 1}, we recall notions of Higgs bundles, HYM connections and equivalent vector bundles.
In \textbf{Section 2}, we derive the doubly-coupled $\tau$-vortex equations from the $SU\left(2\right)$-invariant HYM equations via dimensional reduction. We also introduce a Yang-Mills-type functional whose minima is attained at the solution to the doubly-coupled $\tau$-vortex equations.
In \textbf{Section 3}, we introduce the notion of Higgs pairs and the $\tau$-vortex equations for Higgs bundles as a special case of doubly-coupled $\tau$-vortex equations. Then we formulate the KH correspondence for Higgs pairs.
In \textbf{Section 4}, we construct elliptic complexes $\left(\mathscr{C}^*\right)$ and $\left(\mathscr{B}^*\right)$ arising from the infinitesimal deformations of Higgs pairs and solutions to the $\tau$-vortex equations for Higgs bundles. If the $0$-th and $2$-nd cohomology groups of these complexes vanish, the moduli spaces of Higgs pairs and of solutions to the $\tau$-vortex equations for Higgs bundles are smooth complex manifolds.
In \textbf{Section 5}, under an assumption that the base space of fixed vector bundle is a Riemann surface, we analyze the moduli space of stable Higgs pairs. This section contains the main result of this paper.
\section{Preliminaries}
Throughout this paper, we assume that complex manifolds are connected. Let $M$ be a complex manifold and $E$ be a smooth complex vector bundle of rank $r$ over $M$. First, we recall the definitions of Higgs bundles and the Hermitian-Yang-Mills (HYM) connections.
\subsection{Higgs bundles}
\begin{definition}\cite{Simpson2}
Let $D^{\prime\prime}:\Omega^0\left(E\right)\to\Omega^1\left(E\right)$ be a $\mathbb{C}$-linear mapping. A pair $\left(E,D^{\prime\prime}\right)$ is called a \textit{Higgs bundle} over $M$ if $D^{\prime\prime}$ satisfies the Leibniz rule
\begin{equation}
    D^{\prime\prime}\left(fs\right)=fD^{\prime\prime}s+s\otimes\bar{\partial}f \label{eq:Leibniz rule}
\end{equation}
for $f\in C^\infty\left(M\right),s\in\Omega^0\left(E\right)$ and the integrability condition
\begin{equation}
    D^{\prime\prime}\circ D^{\prime\prime}=0:\Omega^0\left(E\right)\to\Omega^2\left(E\right). \label{eq:integrability condition}
\end{equation}
A triple $\left(E,D^{\prime\prime},h\right)$ is called a \textit{Hermitian Higgs bundle} over $M$ if $\left(E,D^{\prime\prime}\right)$ is a Higgs bundle over $M$ and $h$ is a smooth Hermitian metric on $E$.
\end{definition}

According to the decomposition $\Omega^1\left(E\right)=\Omega^{1,0}\left(E\right)\oplus\Omega^{0,1}\left(E\right)$, we write the structure of Higgs bundle in $E$ as $D^{\prime\prime}=\bar{\partial}^E+\theta$. Then by the condition \eqref{eq:integrability condition}, $\bar{\partial}^E$ defines a holomorphic structure in $E$ and $\theta$ is an $\textrm{End}\left(E\right)$-valued holomorphic $1$-form on $M$ satisfying $\theta\wedge\theta=0$. The $1$-form $\theta$ is called a \textit{Higgs field}.

Let $\left(E,D^{\prime\prime}=\bar{\partial}^E+\theta,h\right)$ be a Hermitian Higgs bundle over $M$ and $\nabla_h^E$ be the Chern connection on a holomorphic Hermitian vector bundle $\left(E,\bar{\partial}^E,h\right)$. We decompose $\nabla_h^E$ as $\nabla_h^E=\partial_h^E+\bar{\partial}^E$. We define an operator $D_h^\prime:\Omega^0\left(E\right)\to\Omega^1\left(E\right)$ by $D_h^\prime=\partial_h^E+\theta_h^*$, where $\theta_h^*\in\Omega^{0,1}\bigl(\textrm{End}\left(E\right)\bigr)$ is the adjoint of $\theta$ with respect to $h$. Then $D_h^\prime$ satisfies the integrability condition $D_h^\prime\circ D_h^\prime=0:\Omega^0\left(E\right)\to\Omega^2\left(E\right)$. Also, we define an operator $D_h:\Omega^0\left(E\right)\to\Omega^1\left(E\right)$ as $D_h=D_h^\prime+D^{\prime\prime}$. This is a connection on $E$ and is called the \textit{Hitchin-Simpson (HS) connection} on the Hermitian Higgs bundle $\left(E,D^{\prime\prime},h\right)$ \cite[p.$4$]{Bruzzo}. By the integrability conditions for $D_h^\prime$ and $D^{\prime\prime}$, the curvature $R\left(D_h\right)$ of $D_h$ is given by
\begin{equation}
  R\left(D_h\right)=R\left(\nabla_h^E\right)+\left[\theta\wedge\theta_h^*\right]+\left(\partial_h^{\textrm{End}\left(E\right)}\theta+\bar{\partial}^{\textrm{End}\left(E\right)}\theta_h^*\right).
\end{equation}
The HS connection on a Hermitian Higgs bundle is considered as a generalization of the Chern connection on a holomorphic Hermitian vector bundle.
\subsection{Hermitian-Yang-Mills Connection}
Let $\left(M,g\right)$ be an $n$-dimensional compact Kähler manifold, $\left(E,D^{\prime\prime}=\bar{\partial}^E+\theta,h\right)$ be a Hermitian Higgs bundle of rank $r$ over $M$ and $\omega$ be the K\"{a}hler form of $\left(M,g\right)$.
\begin{definition} \cite{Simpson2}
  Consider the following equation for the HS connection $D_h$ on $\left(E,D^{\prime\prime},h\right)$:
  \begin{equation}
    \sqrt{-1}\Lambda R\left(D_h\right)=c\,\textrm{id}_E,\quad c=\frac{2\pi\mu\left(E\right)}{\textrm{Vol}\left(M,g\right)} \label{eq:1.5}
  \end{equation}
  This equation is called the \textit{Hermitian-Yang-Mills (HYM) equation}. When $D_h$ satisfies this equation, $h$ is called a \textit{HYM metric} and $D_h$ is called a \textit{HYM connection} on $\left(E,D^{\prime\prime}\right)$.
\end{definition}

\begin{remark}
  Since $\left(M,g\right)$ is a compact Kähler manifold, the constant $c$ is determined by $M$, $g$, and $E$, and does not depend on $h$ and $D_h$.
\end{remark}

We decompose $D_h$ as $D_h=\nabla_h^E+\Theta_h$, where $\nabla_h^E$ is the Chern connection on a holomorphic vector bundle $\left(E,\bar{\partial}^E,h\right)$ and $\Theta_h=\theta+\theta_h^*\in\Omega^1\bigl(\textrm{Herm}\left(E,h\right)\bigr)$. Then if $\theta=0$, the HYM equation coincides with the Hermitian-Einstein (HE) equation. Thus the HYM equation can be considered as a generalization of the HE equation.
\subsection{Equivalent Hermitian vector bundles}\label{subsec:equiv.Herm.vec.bdl}
Next, we recall the notion of equivalent vector bundles. Let $M$ be a smooth manifold, $\left(E,h\right)$ be a smooth Hermitian vector bundle over $M$ and $G$ be a compact Lie group smoothly acting on $M$. $\textrm{Diff}\left(M\right), U\left(E,h\right)$ and $\mathcal{G}\left(E,h\right)$ denote the group of diffeomorphisms of $M$, the group of unitary bundle automorphisms of $\left(E,h\right)$ and the gauge group of $\left(E,h\right)$, respectively. Then we have an exact sequence
\[\xymatrix{\left\{\textrm{id}_E\right\}\ar[r]&\mathcal{G}\left(E,h\right)\ar[r]&U\left(E,h\right)\ar[r]^-{\pi}&\textrm{Diff}\left(M\right),}\]
where $\pi$ is defined so that for $g\in U\left(E,h\right)$, the following diagram commutes
\[\xymatrix{E\ar[r]^-{g}\ar[d]_-{\pi_E}&E\ar[d]^{\pi_E}\\
M\ar[r]^-{\pi\left(g\right)}&M.}\]
Let $\mathscr{H}$ be the preimage of $G \subset\textrm{Diff}\left(M\right)$ under $\pi$.
\begin{definition}\cite{Garcia-Prada}
  Given a homomorphism $\varphi:G\to\mathscr{H}$, a smooth Hermitian vector bundle $\left(E,h\right)$ is called a \textit{$G$-equivalent Hermitian vector bundle} with respect to $\varphi$ if the sequence
  \[\xymatrix{\left\{\textrm{id}_E\right\}\ar[r]&\mathcal{G}\left(E,h\right)\ar[r]&\mathscr{H}\ar[r]&G\ar[r]&\left\{\textrm{id}_M\right\}}\]
  is exact and split and the splitting is compatible with respect to $\varphi$. Furthermore, the metric $h$ is called the \textit{$G$-invariant Hermitian metric} on $E$.
\end{definition}
In the following, for a $G$-equivalent Hermitian vector bundle $\left(E,h\right)$ and the left-action $\varphi:G\to\mathscr{H}$, we write $g\cdot v$ for $\varphi\left(g\right)\left(v\right)$ for $g\in G$ and $v\in E$.

Let $\left(E,h\right)$ be a $G$-equivalent Hermitian vector bundle. The action of $G$ on $\left(E,h\right)$ induces the left-action on $\Omega^0\left(E\right)$ and right-action on $\mathscr{A}\left(E,h\right)$ and $\mathcal{G}\left(E,h\right)$. For $\gamma\in G$ and $s\in\Omega^0\left(E\right)$, the action on $\Omega^0\left(E\right)$ is given by
\[\left(\gamma\cdot s\right)_p:=\gamma\cdot\left(s_{\gamma^{-1}\cdot p}\right),\quad p\in M.\]
Similarly, for $\gamma\in G,\nabla\in\mathscr{A}\left(E,h\right)$ and $g\in\mathcal{G}\left(E,h\right)$, the action on $\mathscr{A}\left(E,h\right)$ and $\mathcal{G}\left(E,h\right)$ are given by
\begin{empheq}[left=\empheqlbrace]{align}
  &\left(\nabla\cdot\gamma\right)_X\left(s\right):=\gamma^{-1}\cdot\nabla_{d\gamma\left(X\right)}\left(\gamma\cdot s\right),\quad s\in\Omega^0\left(E\right),X\in\mathfrak{X}\left(M\right),\\
  &\left(g\cdot\gamma\right)\left(s\right):=\gamma^{-1}\cdot g\left(\gamma\cdot s\right),\quad s\in\Omega^0\left(E\right).
\end{empheq}
These actions of $G$ induce the right-action on the quotient space ${\mathscr{A}\left(E,h\right)}/{\mathcal{G}\left(E,h\right)}$.

For a connection $D$ on $E$, we can decompose it as $D=\nabla_h+\Theta_h$ uniquely, where $\nabla_h\in\mathscr{A}\left(E,h\right)$ and $\Theta_h\in\Omega^1\left(\textrm{Herm}\left(E,h\right)\right)$. $G$ acts on $\Omega^1\left(\textrm{Herm}\left(E,h\right)\right)$ from the right by, for $\gamma\in G$ and $\Theta\in\Omega^1\left(\textrm{Herm}\left(E,h\right)\right)$,
\[\left(\Theta\cdot\gamma\right)\left(s\right)\left(X\right):=\gamma^{-1}\cdot\Bigl(\Theta\left(\gamma\cdot s\right)\bigl(d\gamma\left(X\right)\bigr)\Bigr),\quad s\in\Omega^0\left(E\right),X\in\mathfrak{X}\left(M\right).\]
Therefore $G$ acts on the space $\mathscr{A}\left(E\right)$ of connections on $E$ from the right.

In this paper, we consider an $SU\left(2\right)$-equivalent Hermitian vector bundle over $M\times\mathbb{CP}^1$. Since the complex projective line $\mathbb{CP}^1$ is isomorphic to ${SU\left(2\right)}/{U\left(1\right)}$, a compact Lie group $SU\left(2\right)$ acts on $\mathbb{CP}^1$ from the right smoothly. Suppose that $SU\left(2\right)$ acts on $M$ trivially, then it acts on $M\times\mathbb{CP}^1$.
\begin{proposition}\cite{Garcia-Prada}\label{thm:SU(2)-decomposition}
  \textit{Let $\left(F,h\right)$ be an $SU\left(2\right)$-equivalent Hermitian vector bundle over $M\times\mathbb{CP}^1$.
  \begin{enumerate}
    \item $F$ can be decomposed as
    \begin{equation}\label{SU(2)-equiv..vec.bdl_decomposition}
      F=\bigoplus_kF_k,\quad F_k=p^*E_k\otimes q^*O\left(n_k\right),
    \end{equation}
    where $p$ and $q$ are the projections to $M$ and $\mathbb{CP}^1$, respectively, $E_k$ is a smooth complex vector bundle over $M$ and $n_k\in\mathbb{Z}$ are all different. This decomposition is unique up to isomorphism.
    \item The vector bundles $F_k$ given in \eqref{SU(2)-equiv..vec.bdl_decomposition} are $SU\left(2\right)$-invariantly orthogonal to each other.
    \item Let $\widetilde{h}_k$ be the $SU\left(2\right)$-invariant Hermitian metric on $F_k$ induced by $h$ and $h_k^\prime$ be a standard Hermitian metric on $O\left(n_k\right)$. Then there are Hermitian metrics $h_k$ on $E_k$ such that
    \[\widetilde{h}_k=p^*h_k\otimes q^*h_k^\prime\]
    holds.
  \end{enumerate}}
\end{proposition}
Let $\left(F,h\right)$ be an $SU\left(2\right)$-equivalent Hermitian vector bundle over $M\times\mathbb{CP}^1$ and we consider the decomposition
\begin{equation}\label{eq:decomposition}
  F=\bigoplus_{k=1}^mF_k=\bigoplus_{k=1}^mp^*E_k\otimes q^*O\left(n_k\right)
\end{equation}
given in \eqref{SU(2)-equiv..vec.bdl_decomposition}. From theorem \ref{thm:SU(2)-decomposition} $\left(2\right)$, we have a decomposition
\[h=\bigoplus_{k=1}^m\widetilde{h}_k,\]
where $\widetilde{h}_k$ are $SU\left(2\right)$-invariant Hermitian metrics on $F_k$. Let $\nabla_F$ be a metric connection on $F$ with respect to $h$. Since we have the decompositions
\[\Omega^p\left(F\right)=\bigoplus_{k=1}^m\Omega^p\left(F_k\right),\quad p\in\mathbb{Z}_{\geq 0},\]
a connection $\nabla^F$ is decomposed as
\begin{equation}
  \nabla^F=\begin{pmatrix}
    \beta_{11}&\cdots&\beta_{1m}\\
    \vdots&\ddots&\vdots\\
    \beta_{m1}&\cdots&\beta_{mm}
  \end{pmatrix},
\end{equation}
where for $i,j\in\left\{1,\hdots,m\right\},\beta_{ij}$ is a $\mathbb{C}$-linear map from $\Omega^0\left(F_j\right)$ to $\Omega^1\left(F_i\right)$.
\begin{proposition}\cite{Garcia-Prada}\label{prop:connection_decomposition}
  \textit{\begin{enumerate}
    \item For $i\in\left\{1,\hdots,m\right\},\beta_{ii}$ is a metric connection on $F_i$ with respect to $\widetilde{h}_i$.
    \item For $i,j\in\left\{1,\hdots,m\right\}$ with $i\ne j, \beta_{ij}\in\Omega^1\bigl(\textrm{Hom}\left(F_j,F_i\right)\bigr)$ holds. Moreover, $\beta_{ji}$ is an adjoint of $-\beta_{ij}$ with respect to $\widetilde{h}_i$ and $\widetilde{h}_j$.
  \end{enumerate}}
\end{proposition}
Similarly, for a $\Theta^F\in\Omega^1\bigl(\textrm{Herm}\left(F,h\right)\bigr)$, we decompose it as
\begin{equation}\label{Theta_decomposition}
  \Theta^F=\begin{pmatrix}
    \Theta_{11}&\cdots&\Theta_{1m}\\
    \vdots&\ddots&\vdots\\
    \Theta_{m1}&\cdots&\Theta_{mm}
  \end{pmatrix},
\end{equation}
where for $i,j\in\left\{1,\hdots,m\right\},\Theta_{ij}\in\Omega^1\bigl(\textrm{Hom}\left(F_j,F_i\right)\bigr)$. Then by a direct computation, we can check that for $i,j\in\left\{1,\hdots,m\right\},\Theta_{ji}$ is an adjoint of $\Theta_{ij}$ with respect to $\widetilde{h}_i$ and $\widetilde{h}_j$.
\subsection{Equivalent complex vector bundle}
Let $M$ be a complex manifold, $E$ be a smooth complex vector bundle over $M$ and $G$ be a compact Lie group holomorphically acting on $M$. $\textrm{Aut}_h\left(M\right), GL\left(E\right)$ and $\mathcal{G}\left(E\right)$ denote the group of biholomorphic map of $M$, the group of bundle automorphisms of $E$ and the gauge group of $E$, respectively. Let $\textrm{Aut}_h\left(E\right)$ be the preimage of $\textrm{Aut}_h\left(M\right)\subset\textrm{Diff}\left(M\right)$ under $\pi:GL\left(E\right)\to\textrm{Diff}\left(M\right)$. Then we have an exact sequence
\[\xymatrix{\left\{\textrm{id}_E\right\}\ar[r]&\mathcal{G}\left(E\right)\ar[r]&\textrm{Aut}_h\left(E\right)\ar[r]^-{\pi}&\textrm{Aut}_h\left(M\right).}\]
Let $\mathscr{H}^c$ be the preimage of $G \subset\textrm{Diff}\left(M\right)$ under $\pi|_{\textrm{Aut}_h\left(E\right)}$.
\begin{definition}\cite{Garcia-Prada}
  Given a homeomorphism $\varphi:G\to\mathscr{H}^c$, a smooth complex vector bundle $E$ is called a \textit{$G$-equivalent vector bundle} with respect to $\varphi$ if the sequence
  \[\xymatrix{\left\{\textrm{id}_E\right\}\ar[r]&\mathcal{G}\left(E\right)\ar[r]&\mathscr{H}^c\ar[r]&G\ar[r]&\left\{\textrm{id}_M\right\}}\]
  is exact and split and the splitting is compatible with respect to $\varphi$.
\end{definition}
Let $E$ be a $G$-equivalent vector bundle and $\mathscr{H}^{\prime\prime}\left(E\right)$ denotes the space of holomorphic structures in $E$. The action of $G$ on $E$ induces the right-action on $\mathscr{H}^{\prime\prime}\left(E\right)$ and $\mathcal{G}\left(E\right)$. These are given by, for $\gamma\in G,\bar{\partial}^E\in\mathscr{H}^{\prime\prime}\left(E\right),g\in\mathcal{G}\left(E\right)$,
\begin{empheq}[left=\empheqlbrace]{align}
  &\left(\bar{\partial}^E\cdot\gamma\right)_X\left(s\right):=\gamma^{-1}\cdot\bar{\partial}^E_{d\gamma\left(X\right)}\left(\gamma\cdot s\right),\quad s\in\Omega^0\left(E\right),X\in\mathfrak{X}^{0,1}\left(M\right),\\
  &\left(g\cdot\gamma\right)\left(s\right):=\gamma^{-1}\cdot g\left(\gamma\cdot s\right),\quad s\in\Omega^0\left(E\right).
\end{empheq}
\section{Doubly-Coupled Vortex Equations}
Let $M$ be a compact K\"{a}hler manifold and $\left(E_1,h_1\right),\left(E_2,h_2\right)$ be smooth Hermitian vector bundles over $M$. We consider a smooth Hermitian vector bundle
\[F:=p^*E_1\oplus p^*E_2\otimes q^*O\left(2\right), \quad h:=p^*h_1\oplus p^*h_2\otimes q^*h^\prime\]
over $M\times\mathbb{P}^1$, where $h^\prime$ is a standard Hermitian metric on $O\left(2\right)$. $\left(F,h\right)$ is an $SU\left(2\right)$-equivalent Hermitian vector bundle and $F$ is an $SU\left(2\right)$-equivalent vector bundle since  $SU\left(2\right)$ acts on $\mathbb{CP}^1$ from the right holomorphically.

In this section, we derive the doubly-coupled $\tau$-vortex equations from the $SU\left(2\right)$-invariant HYM equations on $F$ via dimensional reduction. This argument was developed by Ono \cite{Ono} for the case where $M$ is a compact Riemann surface. Since in his argument, the assumption that $M$ is a compact Riemann surface is not essentially used, we provide only a brief explanation and refer the reader to \cite{Ono} for further details. Moreover, we introduce a Yang-Mills-Higgs type functional whose minima corresponds to the solution to the doubly-coupled $\tau$-vortex equations.
\subsection{Invariant connections and correspondence}
Let $\bar{\partial}^F$ be an $SU\left(2\right)$-invariant holomorphic structure on $F$ and $\nabla^F$ be the Chern connection on $\left(F,h\right)$. By Proposition \ref{prop:connection_decomposition}, $\nabla^F$ is decomposed as
\begin{gather}
  \nabla^F=\begin{pmatrix}
    \nabla^{F_1}&\beta_{12}\\
    -\beta_{12}^*&\nabla^{F_2}
  \end{pmatrix},\\
  \nabla^{F_1}\in\mathscr{A}\left(F_1,p^*h_1\right),\nabla^{F_2}\in\mathscr{A}\left(F_2,p^*h_2\otimes q^*h^\prime\right),\beta_{12}\in\Omega^1\bigl(\textrm{Hom}\left(F_2,F_1\right)\bigr).
\end{gather}
Since $\nabla^F$ is an $SU\left(2\right)$-invariant metric connection on $\left(F,h\right)$, the following proposition holds.
\begin{proposition}\cite{Garcia-Prada}\label{prop:F_decomposition}
  \textit{$\bar{\partial}^\prime$ denotes the standard holomorphic structure in $O\left(2\right)$ and $\nabla^\prime$ denotes the Chern connection on $\left(O\left(2\right),\bar{\partial}^\prime,h^\prime\right)$.}
  
  \textit{\begin{enumerate}
    \item There exist metric connections $\nabla^{E_1}$ on $\left(E_1,h_1\right)$ and $\nabla^{E_2}$ on $\left(E_2,h_2\right)$ such that
    \begin{equation}
      \nabla^{F_1}=p^*\nabla^{E_1},\quad \nabla^{F_2}=p^*\nabla^{E_2}\otimes\textrm{id}_{q^*O\left(2\right)}+\textrm{id}_{p^*E_2}\otimes q^*\nabla^\prime
    \end{equation}
    holds.
    \item There exist a section $\phi\in\Omega^0\bigl(\textrm{Hom}\left(E_2,E_1\right)\bigr)$ and a $1$-form $\alpha\in\Omega^1\bigl(O\left(-2\right)\bigr)^{SU\left(2\right)}$ such that
    \[\beta_{12}=p^*\phi\otimes q^*\alpha\]
    holds. $\alpha$ is a generator of a vector space $\Omega^1\bigl(O\left(-2\right)\bigr)^{SU\left(2\right)}\simeq\mathbb{C}$ so it is unique up to constant multiplicity.
  \end{enumerate}}
\end{proposition}
The derivative of the projection $p:M\times\mathbb{CP}^1\to M$ is surjective. Thus by Proposition \ref{prop:F_decomposition}, $\nabla^{E_1},\nabla^{E_2}$ and $\phi$ are unique when we fix an element $\alpha\in\Omega^1\bigl(O\left(-2\right)\bigr)^{SU\left(2\right)}$.

Since $\bar{\partial}^F$ is the $\left(0,1\right)$-part of the Chern connection $\nabla^F$, by Proposition \ref{prop:F_decomposition} we can write $\bar{\partial}^F$ as
\[\bar{\partial}^F=\begin{pmatrix}
  p^*\bar{\partial}^{E_1}&p^*\phi\otimes q^*\alpha^{0,1}\\
  -\left(p^*\phi\otimes q^*\alpha^{1,0}\right)^*&p^*\bar{\partial}^{E_2}\otimes\textrm{id}_{q^*O\left(2\right)}+\textrm{id}_{p^*E_2}\otimes q^*\bar{\partial}^\prime
\end{pmatrix},\]
where $\bar{\partial}^{E_1}$ and $\bar{\partial}^{E_2}$ are the $\left(0,1\right)$-part of the connections $\nabla^{E_1}$ and $\nabla^{E_2}$, respectively. Since $\alpha^{1,0}\in\Omega^{1,0}\bigl(O\left(-2\right)\bigr)^{SU\left(2\right)}=\left\{0\right\}$, $\left(\beta_{12}^*\right)^{0,1}$ vanishes. Therefore $\bar{\partial}^F$ is written as
\[\bar{\partial}^F=\begin{pmatrix}
  p^*\bar{\partial}^{E_1}&p^*\phi\otimes q^*\alpha^{0,1}\\
  0&p^*\bar{\partial}^{E_2}\otimes\textrm{id}_{q^*O\left(2\right)}+\textrm{id}_{p^*E_2}\otimes q^*\bar{\partial}^\prime
\end{pmatrix}.\]
Now, $\bar{\partial}^F$ is a holomorphic structure in $F$, that is, $\bar{\partial}^F\circ\bar{\partial}^F=0$ holds. Thus
\begin{align}
  0=\bar{\partial}^F\circ\bar{\partial}^F=&\begin{pmatrix}
    p^*\left(\bar{\partial}^{E_1}\circ\bar{\partial}^{E_1}\right)&p^*\bar{\partial}^{\textrm{Hom}\left(E_2,E_1\right)}\phi\otimes q^*\alpha^{0,1}+p^*\phi\otimes q^*\bar{\partial}^\prime\alpha^{0,1}\\
    0&p^*\left(\bar{\partial}^{E_2}\circ\partial^{E_2}\right)\otimes\textrm{id}_{q^*O\left(2\right)}+\textrm{id}_{p^*E_2}\otimes q^*\left(\bar{\partial}^\prime\circ\bar{\partial}^\prime\right)
  \end{pmatrix}\\
  =&\begin{pmatrix}
    p^*\left(\bar{\partial}^{E_1}\circ\bar{\partial}^{E_1}\right)&p^*\bar{\partial}^{\textrm{Hom}\left(E_2,E_1\right)}\phi\otimes q^*\alpha^{0,1}\\
    0&p^*\left(\bar{\partial}^{E_2}\circ\partial^{E_2}\right)\otimes\textrm{id}_{q^*O\left(2\right)}
  \end{pmatrix}
\end{align}
holds. Since $\alpha$ is type $\left(0,1\right)$ and nowhere-vanishing, $\bar{\partial}^{E_1}$ and $\bar{\partial}^{E_2}$ define the holomorphic structure in $E_1$ and $E_2$, respectively and $\bar{\partial}^{\textrm{Hom}\left(E_2,E_1\right)}\phi=0$ holds.

Let $\theta^F$ be an $SU\left(2\right)$-invariant Higgs field on $\left(F,\bar{\partial}^F\right)$ and $\Theta^F$ denotes $\theta^F+\theta_h^F$. By \eqref{Theta_decomposition}, $\Theta^F$ is decomposed as
\begin{gather}
  \Theta^F=\begin{pmatrix}
    \Theta^{F_1}&\Theta_{12}\\
    \Theta_{12}^*&\Theta^{F_2}
  \end{pmatrix},\\
  \Theta^{F_1}\in\Omega^1\bigl(\textrm{Herm}\left(F_1,p^*h_1\right)\bigr),\Theta^{F_2}\in\Omega^1\bigl(\textrm{Herm}\left(F_2,p^*h_2\otimes q^*h^\prime\right)\bigr),\\
  \Theta_{12}\in\Omega^1\bigl(\textrm{Hom}\left(F_2,F_1\right)\bigr).
\end{gather}
By a similar argument for the metric connection $\nabla^F$, we can show the following proposition.
\begin{proposition}\label{prop:Theta_F_decomposition}
  \textit{\begin{enumerate}
    \item There exist $\Theta^{E_1}\in\Omega^1\bigl(\textrm{Herm}\left(E_1,h_1\right)\bigr)$ and $\Theta^{E_2}\in\Omega^1\bigl(\textrm{Herm}\left(E_2,h_2\right)\bigr)$ such that
    \begin{equation}
      \Theta^{F_1}=p^*\Theta^{E_1},\quad \Theta^{F_2}=p^*\Theta^{E_2}\otimes\textrm{id}_{q^*O\left(2\right)}
    \end{equation}
    holds.
    \item There exist a section $\phi\in\Omega^0\bigl(\textrm{Hom}\left(E_2,E_1\right)\bigr)$ and a $1$-form $\alpha\in\Omega^1\bigl(O\left(-2\right)\bigr)^{SU\left(2\right)}$ such that
    \[\Theta_{12}=p^*\phi\otimes q^*\alpha\]
    holds. $\alpha$ is a generator of a vector space $\Omega^1\bigl(O\left(-2\right)\bigr)^{SU\left(2\right)}\simeq\mathbb{C}$ so it is unique up to constant multiplicity.
  \end{enumerate}}
\end{proposition}
As in the case of the Chern connection $\nabla^F$, $\Theta^{E_1},\Theta^{E_2}$ and $\phi$ are unique when we fix an element $\alpha\in\Omega^1\bigl(O\left(-2\right)\bigr)^{SU\left(2\right)}$.

Since $\theta^F$ is the $\left(1,0\right)$-part of the $1$-form $\Theta^F$, by Proposition \ref{prop:Theta_F_decomposition} we can write $\theta^F$ as
\[\theta^F=\begin{pmatrix}
  p^*\theta^{E_1}&0\\
  p^*\phi^*\otimes q^*\left(\alpha^{0,1}\right)^*&p^*\theta^{E_2}\otimes\textrm{id}_{q^*O\left(2\right)}
\end{pmatrix},\]
where $\theta^{E_1}$ and $\partial^{E_2}$ are the $\left(1,0\right)$-part of $1$-forms $\Theta^{E_1}$ and $\Theta^{E_2}$, respectively. Since $\alpha$ is type $\left(0,1\right)$ and $SU\left(2\right)$-invariant, $\left(\alpha^{0,1}\right)^*=\alpha^*\in\Omega^1\bigl(O\left(2\right)\bigr)$ is also $SU\left(2\right)$-invariant. We set $\psi=\phi^*$ and $\beta=\alpha^*$. Now, $\theta^F$ is a Higgs field in $F$, that is, $\bar{\partial}^F$ and $\theta^F$ satisfy
\begin{empheq}[left=\empheqlbrace]{align}
  &\theta^F\wedge\theta^F=0 \label{eq:integrable},\\
  &\bar{\partial}^{\textrm{End}\left(F\right)}\theta^F=0 \label{eq:holomorphic}.
\end{empheq}
First, from \eqref{eq:integrable},
\begin{align}
  0=\theta^F\wedge\theta^F=&\begin{pmatrix}
    p^*\left(\theta^{E_1}\wedge\theta^{E_1}\right)&0\\
    p^*\theta^{\textrm{Hom}\left(E_1,E_2\right)}\psi\otimes q^*\beta^{1,0}&p^*\left(\theta^{E_2}\wedge\theta^{E_2}\right)\otimes\textrm{id}_{q^*O\left(2\right)}
  \end{pmatrix}
\end{align}
holds. Since $\beta^{1,0}=\beta$ is nowhere-vanishing,
\[\theta^{E_1}\wedge\theta^{E_1}=0,\quad\theta^{E_2}\wedge\theta^{E_2}=0\ \ \textrm{and}\ \ \theta^{\textrm{Hom}\left(E_1,E_2\right)}\psi=0\]
hold. Next, from \eqref{eq:holomorphic},
\begin{align}
  &\bar{\partial}^{\textrm{End}\left(F\right)}\theta^F\\
  =&\bar{\partial}^F\circ\theta^F+\theta^F\circ\bar{\partial}^F\\
  =&\begin{pmatrix}
    p^*\left(\bar{\partial}^{E_1}\circ\theta^{E_1}\right)+p^*\left(\phi\circ\psi\right)\otimes q^*\left(\alpha\wedge\beta\right)&-p^*\left(\phi\circ\theta^{E_2}\right)\otimes q^*\alpha\\
    p^*\left(\bar{\partial}^{E_2}\circ\psi\right)\otimes q^*\beta+p^*\phi\otimes q^*\bar{\partial}^\prime\beta&p^*\left(\bar{\partial}^{E_2}\circ\theta^{E_2}\right)\otimes\textrm{id}_{q^*O\left(2\right)}-p^*\theta^{E_2}\otimes q^*\bar{\partial}^\prime
  \end{pmatrix}\\
  &{\footnotesize+\begin{pmatrix}
    p^*\left(\theta^{E_1}\circ\bar{\partial}^{E_1}\right)&p^*\left(\theta^{E_1}\circ\phi\right)\otimes q^*\alpha\\
    -p^*\left(\psi\circ\bar{\partial}^{E_1}\right)\otimes q^*\beta&p^*\left(\psi\circ\phi\right)\otimes q^*\left(\beta\wedge\alpha\right)+p^*\left(\theta^{E_2}\circ\bar{\partial}^{E_2}\right)\otimes\textrm{id}_{q^*O\left(2\right)}+p^*\theta^{E_2}\otimes q^*\bar{\partial}^\prime
  \end{pmatrix}}\\
  =&{\scriptsize\begin{pmatrix}
    p^*\left(\bar{\partial}^{\textrm{End}\left(E_1\right)}\theta^{E_1}\right)+p^*\left(\phi\circ\psi\right)\otimes q^*\left(\alpha\wedge\beta\right)&p^*\left(\theta^{\textrm{Hom}\left(E_2,E_1\right)}\phi\right)\otimes q^*\alpha\\
    p^*\left(\bar{\partial}^{\textrm{Hom}\left(E_1,E_2\right)}\psi\right)\otimes q^*\beta&p^*\left(\psi\circ\phi\right)\otimes q^*\left(\beta\wedge\alpha\right)+p^*\left(\bar{\partial}^{\textrm{End}\left(E_2\right)}\theta^{E_2}\right)\otimes\textrm{id}_{q^*O\left(2\right)}
  \end{pmatrix}}
\end{align}
holds since $\beta\in\Omega^1\bigl(O\left(2\right)\bigr)$ is holomorphic. This implies that $\theta^{E_1}$ and $\theta^{E_2}$ are holomorphic with respect to the holomorphic structures $\bar{\partial}^{E_1}$ and $\bar{\partial}^{E_2}$, respectively and $\phi$ and $\psi$ satisfy
\begin{gather}
  \theta^{\textrm{Hom}\left(E_2,E_1\right)}\phi=0,\quad \theta^{\textrm{Hom}\left(E_1,E_2\right)}\psi=0,\\
  \phi\circ\psi=0,\quad \psi\circ\phi=0.
\end{gather}

We define subsets $\widetilde{\mathscr{N}}$ and $\mathscr{N}$ of $\mathscr{H}\left(E_1,h_1\right)\times\mathscr{H}\left(E_2,h_2\right)\times\Omega^0\bigl(\textrm{Hom}\left(E_2,E_1\right)\bigr)\times\Omega^0\bigl(\textrm{Hom}\left(E_1,E_2\right)\bigr)$ as
\[\widetilde{\mathscr{N}}:=\left\{\left(D_{E_1},D_{E_2},\phi,\psi\right)\middle|\phi\circ\psi=0,\psi\circ\phi=0\right\}\]
and
\[\mathscr{N}:=\left\{\left(D_{E_1},D_{E_2},\phi,\psi\right)\in\widetilde{\mathscr{N}}\middle|D_{\textrm{Hom}\left(E_2,E_1\right)}^{\prime\prime}\phi=0,D_{\textrm{Hom}\left(E_1,E_2\right)}^{\prime\prime}\psi=0\right\}.\]
Then we have a one-to-one correspondence between $\mathscr{H}\left(F,h\right)^{SU\left(2\right)}$ and $\mathscr{N}$.
\subsection{Dimensional reduction of the Hermitian-Yang-Mills equation}
We take $\tau\in\mathbb{R}$ so that a real number $\sigma$, defined below, is positive.
\[\sigma:=\frac{2\rank E_2}{\displaystyle\frac{\rank E_1+\rank E_2}{4\pi}\tau-\frac{\deg E_1+\deg E_2}{\textrm{Vol}\left(M,g\right)}}\]
Also, we define $\tau^\prime\in\mathbb{R}$ so that
\begin{equation}\label{eq:tau-prime}
  \tau\rank E_1+\tau^\prime\rank E_2=\frac{4\pi}{\textrm{Vol}\left(M,g\right)}\left(\deg E_1+\deg E_2\right)
\end{equation}
holds. Let $g_{FS}$ be the Fubini-Study metric on $\mathbb{CP}^1$ with $\textrm{Vol}\left(\mathbb{CP}^1,g_{FS}\right)=1$. Then
\[g_\sigma:=p^*g\oplus\sigma q^*g_{FS}\]
is an $SU\left(2\right)$-invariant K\"{a}hler metric on $M\times\mathbb{CP}^1$. $\omega_\sigma$ denotes the K\"{a}hler form of $g_\sigma$.
\begin{proposition}\label{prop:dimensional reduction}
  \textit{We choose $\alpha\in\Omega^1\bigl(O\left(-2\right)\bigr)^{SU\left(2\right)}$ and $\beta\in\Omega^1\bigl(O\left(2\right)\bigr)^{SU\left(2\right)}$ so that $\beta=\alpha^*$ and $\displaystyle\alpha\wedge\beta=\frac{\sqrt{-1}}{2}\sigma\omega_{FS}$ holds. For a quadruplet $\left(D_{E_1},D_{E_2},\phi,\psi\right)\in\mathscr{N}$ and a connection $D\in\mathscr{H}\left(F,h\right)^{SU\left(2\right)}$ corresponding to it, the following are equivalent.
  \begin{enumerate}
    \item $D$ is a Hermitian-Yang-Mills connection with respect to $g_\sigma$.
    \item $\left(D_{E_1},D_{E_2},\phi,\psi\right)$ satisfies the equations
    \begin{empheq}[left=\empheqlbrace]{align}
      \sqrt{-1}\Lambda_gR\left(D_{E_1}\right)+\frac{1}{2}\phi\circ\phi^*-\frac{1}{2}\psi^*\circ\psi=\frac{\tau}{2}\textrm{id}_{E_1}, \label{eq:dcv_equation_1}\\
      \sqrt{-1}\Lambda_gR\left(D_{E_2}\right)-\frac{1}{2}\phi^*\circ\phi+\frac{1}{2}\psi\circ\psi^*=\frac{\tau^\prime}{2}\textrm{id}_{E_2} \label{eq:dcv_equation_2}
    \end{empheq}
  \end{enumerate}}
\end{proposition}
\begin{proof}
  From Proposition \ref{prop:F_decomposition} and \ref{prop:Theta_F_decomposition}, we can decompose $D$ as
  \[D=\begin{pmatrix}
    p^*D_{E_1}&p^*\left(\phi+\psi^*\right)\otimes q^*\alpha\\
    p^*\left(-\phi^*+\psi\right)\otimes q^*\beta&p^*D_{E_2}\otimes\textrm{id}_{O\left(2\right)}+\textrm{id}_{E_2}\otimes q^*\nabla^\prime
  \end{pmatrix}.\]
  Thus the curvature $R\left(D\right)$ of $D$ is given by
  \begin{align}\label{eq:curvature representation}
    &R\left(D\right)=\\
    &{\footnotesize\begin{pmatrix}
    R\left(D_{F_1}\right)+p^*\left(-\phi\circ\phi^*+\psi^*\circ\psi\right)\otimes q^*\left(\alpha\wedge\beta\right)&D_{\textrm{Hom}\left(F_2,F_1\right)}\left(p^*\left(\phi+\psi^*\right)\otimes q^*\alpha\right)\\
    D_{\textrm{Hom}\left(F_1,F_2\right)}\left(p^*\left(-\phi^*+\psi\right)\otimes q^*\beta\right)&R\left(D_{F_2}\right)+p^*\left(-\phi^*\circ\phi+\psi\circ\psi^*\right)\otimes q^*\left(\beta\wedge\alpha\right)
  \end{pmatrix}},
\end{align}
  where $F_1=p^*E_1$ and $F_2=p^*E_2\otimes q^*O\left(2\right)$. Suppose that $D$ is a Hermitian-Yang-Mills connection with respect to $g_\sigma$. Then we have
  \begin{empheq}[left=\empheqlbrace]{align}
    &\sqrt{-1}\Lambda_{g_\sigma}\left(R\left(D_{F_1}\right)+\frac{\sqrt{-1}}{2}p^*\left(-\phi\circ\phi^*+\psi^*\circ\psi\right)\otimes\sigma q^*\omega_{FS}\right)=c\,\textrm{id}_{p^*E_1}\\
    &\sqrt{-1}\Lambda_{g_\sigma}\left(R\left(D_{F_2}\right)+\frac{\sqrt{-1}}{2}p^*\left(\phi^*\circ\phi-\psi\circ\psi^*\right)\otimes\sigma q^*\omega_{FS}\right)=c\,\textrm{id}_{p^*E_2}.
  \end{empheq}
  These equations are equivalent to the equations
  \begin{empheq}[left=\empheqlbrace]{align}
    &\sqrt{-1}\Lambda_gR\left(D_{E_1}\right)+\frac{1}{2}\phi\circ\phi^*-\frac{1}{2}\psi^*\circ\psi=c\,\textrm{id}_{E_1}\\
    &\sqrt{-1}\Lambda_gR\left(D_{E_2}\right)-\frac{1}{2}\phi^*\circ\phi+\frac{1}{2}\psi\circ\psi^*=\left(c-\frac{4\pi}{\sigma}\right)\textrm{id}_{E_2}.
  \end{empheq}
  Now, $\displaystyle c=\frac{2\pi}{\textrm{Vol}\left(M,g\right)\sigma}\cdot\frac{\deg_{g_\sigma}F}{\rank F}$ holds and $\deg_{g_\sigma}F$ is given by
  \[\deg_{g_\sigma}F=\sigma\left(\deg E_1+\deg E_2\right)+2\rank E_2\cdot\textrm{Vol}\left(M,g\right).\]
  Thus from the definition of $\sigma$, we have $\displaystyle c=\frac{1}{2}\tau$ and $\displaystyle -\frac{4\pi}{\sigma}+c=\frac{1}{2}\tau^\prime$.

  Conversely, suppose that $\left(D_{E_1},D_{E_2},\phi,\psi\right)$ satisfies equations \eqref{eq:dcv_equation_1} and \eqref{eq:dcv_equation_2}. To show that $D$ is a Hermitian-Yang-Mills connection, it is sufficient to prove that
  \begin{equation}\label{eq:off-diagonal}
    \Lambda D_{\textrm{Hom}\left(F_2,F_1\right)}\left(p^*\left(\phi+\psi^*\right)\otimes q^*\alpha\right)=0\ \ \textrm{and}\ \ \Lambda D_{\textrm{Hom}\left(F_1,F_2\right)}\left(p^*\left(-\phi^*+\psi\right)\otimes q^*\beta\right)=0
  \end{equation}
  holds. Since $\alpha$ is anti-holomorphic and $\beta$ is holomorphic, we have
  \begin{align}
    &D_{\textrm{Hom}\left(F_2,F_1\right)}\left(p^*\left(\phi+\psi^*\right)\otimes q^*\alpha\right)=p^*\left(D^\prime\phi+D^{\prime\prime}\psi^*\right)\otimes q^*\alpha+p^*\left(\phi+\psi^*\right)\otimes q^*\bar{\partial}^\prime\alpha\\
    &D_{\textrm{Hom}\left(F_1,F_2\right)}\left(p^*\left(-\phi^*+\psi\right)\otimes q^*\beta\right)=p^*\left(-D^{\prime\prime}\phi^*+D^\prime\psi^*\right)\otimes q^*\alpha+p^*\left(\phi+\psi^*\right)\otimes q^*\partial^\prime\beta
  \end{align}
  and thus \eqref{eq:off-diagonal} holds.
\end{proof}
\begin{definition}
  For a quadruplet $\left(D_{E_1},D_{E_2},\phi,\psi\right)\in\mathscr{N}$, the system of equations \eqref{eq:dcv_equation_1} and \eqref{eq:dcv_equation_2} are called the \textit{doubly-coupled $\tau$-vortex equations}.
\end{definition}
\subsection{Yang-Mills-Higgs-type functional}
By Theorem \ref{prop:dimensional reduction}, the Hermitian-Yang-Mills connection with respect to $g_\sigma$ corresponds to the solution to the doubly-coupled $\tau$-vortex equation. In this subsection, we discuss this correspondence from the perspective of Yang-Mills-type functional. For real numbers $\tau$ and $\tau^\prime$, consider a functional $\textrm{YMH}_{\tau,\tau^\prime}:\widetilde{\mathscr{N}}\to\mathbb{R}_{\geq 0}$ defined by, for $\left(D_{E_1},D_{E_2},\phi,\psi\right)\in\widetilde{\mathscr{N}}$,
\begin{align}
  \textrm{YMH}_{\tau,\tau^\prime}&\left(D_{E_1},D_{E_2},\phi,\psi\right)\\
  :=&\left\|R\left(D_{E_1}\right)\right\|_{L^2\left(M\right)}^2+\left\|R\left(D_{E_2}\right)\right\|_{L^2\left(M\right)}^2\\
  &+\left\|D_{\textrm{Hom}\left(E_2,E_1\right)}\phi\right\|_{L^2\left(M\right)}^2+\left\|D_{\textrm{Hom}\left(E_1,E_2\right)}\psi\right\|_{L^2\left(M\right)}^2\\
  &+\frac{1}{4}\left\|\phi\circ\phi^*-\psi^*\circ\psi-\tau\,\textrm{id}_{E_1}\right\|_{L^2\left(M\right)}^2+\frac{1}{4}\left\|\psi\circ\psi^*-\phi^*\circ\phi-\tau^\prime\,\textrm{id}_{E_2}\right\|_{L^2\left(M\right)}^2.
\end{align}
By a direct computation as in \cite[Lemma $2.10$]{Garcia-Prada}, we obtain the following proposition.
\begin{proposition}\label{prop:functional_minima}
  \textit{\begin{align}
    \textrm{YMH}_{\tau,\tau^\prime}&\left(D_{E_1},D_{E_2},\phi,\psi\right)\\
    =&\left\|\sqrt{-1}\Lambda R\left(D_{E_1}\right)+\frac{1}{2}\phi\circ\phi^*-\frac{1}{2}\psi^*\circ\psi-\frac{\tau}{2}\textrm{id}_{E_1}\right\|_{L^2\left(M\right)}^2\\
    &+\left\|\sqrt{-1}\Lambda R\left(D_{E_2}\right)-\frac{1}{2}\phi^*\circ\phi+\frac{1}{2}\psi\circ\psi^*-\frac{\tau^\prime}{2}\textrm{id}_{E_2}\right\|_{L^2\left(M\right)}^2\\
    &+2\left\|D_{\textrm{Hom}\left(E_2,E_1\right)}^{\prime\prime}\phi\right\|_{L^2\left(M\right)}^2+2\left\|D_{\textrm{Hom}\left(E_1,E_2\right)}^{\prime\prime}\psi\right\|_{L^2\left(M\right)}^2\\
    &+2\pi\left(\tau\deg E_1+\tau^\prime\deg E_2\right)-8\pi^2\left(\textrm{Ch}_2\left(E_1\right)+\textrm{Ch}_2\left(E_2\right)\right)
  \end{align}}
\end{proposition}
Take $D\in\mathscr{H}\left(E,h\right)^{SU\left(2\right)}$ and let $\left(D_{E_1},D_{E_2},\phi,\psi\right)\in\mathscr{N}$ be an element corresponding to $D$. The curvature $R\left(D\right)$ of $D$ is given by \eqref{eq:curvature representation}. Its $L^2$-norm with respect to a metric $g_\sigma$ on $M\times\mathbb{P}^1$ is given by $\textrm{YMH}_{\tau,\tau^\prime}\left(D_{E_1},D_{E_2},\phi,\psi\right)$. Hence $D$ is a Hermitian-Yang-Mills connection if and only if $D$ attains the minimum of the Yang-Mills functional. Thus by Proposition \ref{prop:functional_minima}, $\left(D_{E_1},D_{E_2},\phi,\psi\right)\in\mathscr{N}$ is a solution to the doubly-coupled $\tau$-vortex equations if and only if $D$ is a Hermitian-Yang-Mills connection.
\section{Vortex Equations For Higgs Bundles}
Let $\left(M,g\right)$ be a compact K\"{a}hler manifold and $E$ be a smooth complex vector bundle over $M$. Let $h$ be a smooth Hermitian metric on $E$. $\underline{\mathbb{C}}$ denotes the trivial line bundle over $M$ and $h_{\textrm{std}}$ denotes the standard Hermitian metric on $\underline{\mathbb{C}}$. In this section, we introduce the $\tau$-vortex equations for Higgs bundles as a special case of doubly-coupled $\tau$-vortex equations.
\subsection{Definition of the vortex equations for Higgs bundles}
For a quadruplet $\left(D_E,D_{\underline{\mathbb{C}}},s,\xi\right)\in\mathscr{N}$, either $s$ or $\xi$ vanishes identically since $\textrm{supp}\,s$ and $\textrm{supp}\,\xi$ are disjoint and they are holomorphic. Suppose $\xi$ vanishes. Then if the quadruplet is a solution to the doubly-coupled $\tau$-vortex equations, it satisfies the following equations.
\begin{empheq}[left=\empheqlbrace]{align}
  \sqrt{-1}\Lambda R\left(D_E\right)+\frac{1}{2}s\circ s^*=\frac{\tau}{2}\textrm{id}_E \label{eq:vortex_Higgs},\\
  \sqrt{-1}\Lambda R\left(D_{\underline{\mathbb{C}}}\right)-\frac{1}{2}s^*\circ s=\frac{\tau^\prime}{2}\textrm{id}_{\underline{\mathbb{C}}} \label{eq:linear equation}.
\end{empheq}
\begin{definition}
  \textit{A Higgs pair} is a pair $\left(D^{\prime\prime},s\right)$ consisting of a structure of Higgs bundle $D^{\prime\prime}$ in $E$ and $s\in\Omega^0\left(E\right)$ satisfying $D^{\prime\prime}s=0$. When we fix a Hermitian metric $h$ on $E$, for a Higgs pair $\left(D^{\prime\prime},s\right)$, the equation \eqref{eq:vortex_Higgs} is called the \textit{$\tau$-vortex equation for Higgs bundle $\left(E,D^{\prime\prime}\right)$}.
\end{definition}
If a Higgs pair $\left(D^{\prime\prime},s\right)$ is a solution to $\tau$-vortex equation, the equation \eqref{eq:linear equation} admits a unique solution $D_{\underline{\mathbb{C}}}$ which is compatible with the trivial holomorphic structure in $\underline{\mathbb{C}}$. (cf. \cite{Garcia-Prada}.) Therefore we can identify the $\tau$-vortex equation with a special case of the doubly-coupled $\tau$-vortex equations.
\subsection{Kobayashi-Hitchin correspondence for Higgs pairs}
In $2003$, \'{A}lvarez-C\'{o}nsul and Garc\'{i}a-Prada \cite{Alvarez-Consul-Garcia-Prada} defined quiver bundles over a compact K\"{a}hler manifold and their stability and investigated the Kobayashi-Hitchin correspondence for quiver bundles. A Higgs pair is an example of a quiver bundle. Indeed, let $Q=\left(Q_0,Q_1\right)$ be a quiver with vertices $Q_0:=\left\{0,1\right\}$ and arrows $Q_1:=\left\{a:0 \to 1,b:1 \to 1\right\}$. We set $E_0:=\underline{\mathbb{C}}, E_1:=E, M_a:=\underline{\mathbb{C}}$ and $M_b:=T^\prime M$. Also we set $\phi_a:=s:\underline{\mathbb{C}}\to E$ and $\phi_b$ to be a Higgs field of $E$. In the case of Higgs pairs, the stability for quiver bundles can be written as below. (Substitute $\sigma_0=\sigma_1=2,\tau_0=\tau^\prime,\tau_1=\tau$ into Definition $2.5$ of \cite{Alvarez-Consul-Garcia-Prada}.)
\begin{definition}\label{def:tau-stability}
  Fix a real number $\tau$. A Higgs pair $\left(D^{\prime\prime},s\right)$ is $\tau$-stable if it satisfies the following conditions.
  \begin{enumerate}
    \item $\displaystyle\mu\left(F\right)<\frac{\textrm{Vol}\left(M,g\right)}{4\pi}\tau$ holds for every Higgs subsheaf $F$ of $\left(E,D^{\prime\prime}\right)$ with $0<\rank F$.
    \item $\displaystyle\mu\left(E/F\right)>\frac{\textrm{Vol}\left(M,g\right)}{4\pi}\tau$ holds for every Higgs subsheaf $F$ of $\left(E,D^{\prime\prime}\right)$ with $s\in H^0\left(F\right)$ and $0<\rank F<\rank E$.
  \end{enumerate}
\end{definition}
The Kobayashi-Hitchin correspondence is established as follows.
\begin{theorem}\label{thm:KHC}
  \textit{For a Higgs pair $\left(D^{\prime\prime},s\right)$, the following are equivalent.
  \begin{enumerate}[$\left(1\right)$]
    \item There exists a Hermitian metric $h$ on $E$ such that $\left(D^{\prime\prime},s\right)$ is a solution to the $\tau$-vortex equation for Higgs bundle \eqref{eq:vortex_Higgs}.
    \item $\left(D^{\prime\prime},s\right)$ satisfies either $\left(a\right)$ or $\left(b\right)$.
    \begin{enumerate}[$\left(a\right)$]
      \item $\left(E,D^{\prime\prime}\right)$ is $\tau$-stable.
      \item There exists Higgs subbundles $E_s, E^\prime$ of $\left(E,D^{\prime\prime}\right)$ such that $E$ splits holomorphically as $E=E_s\oplus E^\prime$, where $E_s$ and $E^\prime$ satisfy the following conditions.
      \begin{itemize}
        \item $E^\prime$ is poly-stable Higgs sheaf with $\displaystyle\mu\left(E^\prime\right)=\frac{\textrm{Vol}\left(M,g\right)}{4\pi}\tau$ holds.
        \item $E_s$ contains the section $s$ and a Higgs pair $\left(E_s,s\right)$ is $\tau$-stable.
      \end{itemize}
    \end{enumerate}
  \end{enumerate}}
\end{theorem}
We define a set $\mathscr{T}$ as
\[\mathscr{T}:=\left\{\frac{4\pi}{\textrm{Vol}\left(M,g\right)}\mu\left(F\right)\middle|F \textrm{ is a Higgs subbundle of }E\right\}.\]
Then by Theorem \ref{thm:KHC}, the following corollary holds.
\begin{corollary}
  If we choose $\tau\in\mathbb{R}$ so that $\tau\notin\mathscr{T}$, for a Higgs pair $\left(D^{\prime\prime},s\right)$, it is $\tau$-stable if and only if there exists a Hermitian metric $h$ on $E$ such that $\left(D^{\prime\prime},s\right)$ is a solution to the $\tau$-vortex equation for Higgs bundle.
\end{corollary}
Thus in this paper, we assume the following.
\begin{assumption}\label{eq:tau_exclusion}
  $\tau\notin\mathscr{T}$.
\end{assumption}
For later use in section $5$, we state the following lemmas.
\begin{lemma}\label{lemma:s-nonzero}
  For a $\tau$-stable Higgs pair $\left(D^{\prime\prime},s\right)$, $s$ is not a zero-section.
\end{lemma}
\begin{proof}
  Suppose that $s=0$ holds. By the Kobayashi-Hitchin correspondence for Higgs pairs, there exists a Hermitian metric $h$ on $E$ such that $\left(D^{\prime\prime},s\right)$ is a solution to the equation \eqref{eq:vortex_Higgs}. This equation coincides with the HYM equation. Thus $\displaystyle\tau=\frac{4\pi\mu\left(E\right)}{\textrm{Vol}\left(M,g\right)}$ holds. However, by the definition of $\tau$-stability for Higgs pairs (Definition \ref{def:tau-stability}), we have $\displaystyle\mu\left(E\right)<\frac{\textrm{Vol}\left(M,g\right)}{4\pi}\tau=\mu\left(E\right)$ and this is a contraction. This completes the proof.
\end{proof}
\begin{lemma}\label{lem:L^4-identity}
  \textit{Let $h$ be a Hermitian metric on $E$ and $\left(D^{\prime\prime}=\bar{\partial}^E+\theta,s\right)$ be a solution to the $\tau$-vortex equation for Higgs bundle \eqref{eq:vortex_Higgs}. Then the following holds.
  \begin{enumerate}
    \item $\Delta\left\|s\right\|^2+2\left\|\partial^Es\right\|^2+2\left\|\theta^*s\right\|^2+\left\|s\right\|^4-\tau\left\|s\right\|^2=0.$
    \item $\Delta\left\|\theta\right\|^2+2\left\|\partial \theta\right\|^2+2\left\|\left[\theta\wedge\theta^*\right]\right\|^2+2\left\|s^*\circ\theta\right\|^2+2\textrm{Ric}\left(\theta,\theta\right)=0.$
  \end{enumerate}}
\end{lemma}
\begin{proof}
  We can obtain both equations by appropriately modifying \cite[Lemma $4.1$]{Siqi}.
\end{proof}
\section{Moduli Space}
Let $\left(M,g\right)$ be an $n$-dimensional compact K\"{a}hler manifold, $\omega$ be a K\"{a}hler form of $\left(M,g\right)$ and $E$ be a smooth complex vector bundle over $M$. In this section, we study the local moduli of Higgs pairs and of solutions to the $\tau$-vortex equations for Higgs bundles.
\subsection{Local moduli of Higgs pairs}
We set $\mathscr{HP}\left(E\right)$ to be the space of the structures of Higgs pairs in $E$. The gauge group $\mathcal{G}\left(E\right)$ of $E$ acts on $\mathscr{HP}\left(E\right)$ from the right by, for $\left(D^{\prime\prime},s\right)\in\mathscr{HP}\left(E\right)$ and $g\in\mathcal{G}\left(E\right)$,
\[\left(D^{\prime\prime},s\right)\cdot g:=\left(g^{-1}\circ D^{\prime\prime}\circ g,g^{-1}s\right).\]
$\mathscr{M}_{\textrm{HP}}\left(E\right)$ denotes the quotient space $\mathscr{HP}\left(E\right)/\mathcal{G}\left(E\right)$ and is called the moduli space of Higgs pairs.

For $k\in\mathbb{Z}_{\geq 0}$, $\mathscr{C}^k$ denotes the vector space $\Omega^k\bigl(\textrm{End}\left(E\right)\bigr)\oplus\Omega^{k-1}\left(E\right)$. A Higgs pair $\left(D^{\prime\prime},s\right)\in\mathscr{HP}\left(E\right)$ induces the following elliptic complex.
\begin{equation}
  \left(\mathscr{C}^*\right):\xymatrix{0\ar[r]&\mathscr{C}^0\ar[r]^-{d_0}&\mathscr{C}^1\ar[r]^-{d_1}&\mathscr{C}^2\ar[r]^-{d_2}&\cdots\ar[r]^-{d_{2n}}&\mathscr{C}^{2n}\ar[r]^-{d_{2n+1}}&\mathscr{C}^{2n}\ar[r]&0}
\end{equation}
where $d_k:\mathscr{C}^k\to\mathscr{C}^{k+1}$ is defined by, for $\left(\alpha,\eta\right)\in\mathscr{C}^k=\Omega^k\bigl(\textrm{End}\left(E\right)\bigr)\oplus\Omega^{k-1}\left(E\right)$,
\[d_k\begin{pmatrix}
  \alpha\\
  \eta
\end{pmatrix}:=\begin{pmatrix}
  D_{\textrm{End}\left(E\right)}^{\prime\prime}&0\\
  -ev\left(s\right)&-D^{\prime\prime}
\end{pmatrix}\begin{pmatrix}
  \alpha\\
  \eta
\end{pmatrix}.\]
It is easy to check that a complex $\left(\mathscr{C}^*\right)$ is elliptic. For a Higgs pair $\left(D^{\prime\prime},s\right)$, $H_{\left(D^{\prime\prime},s\right)}^k\left(\mathscr{C}^*\right)$ denotes the $k$-th cohomology group of the complex $\left(\mathscr{C}^*\right)$ induced by $\left(D^{\prime\prime},s\right)$. $\widetilde{\mathscr{C}}^k$ denotes the holomorphic vector bundles $\left(\textrm{End}\left(E\right)\otimes\Lambda^{k,0}M\right)\oplus\left(E\otimes\Lambda^{k-1,0}M\right)$. The $k$-th cohomology group $H_{\left(D^{\prime\prime},s\right)}^k\left(\mathscr{C}^*\right)$ coincides with the $k$-th hypercohomology group of the following complex.
\begin{equation}
  \left(\widetilde{\mathscr{C}}^*\right):\xymatrix{0\ar[r]&\widetilde{\mathscr{C}}^0\ar[r]^-{\widetilde{d}_0}&\widetilde{\mathscr{C}}^1\ar[r]^-{\widetilde{d}_1}&\widetilde{\mathscr{C}}^2\ar[r]^-{\widetilde{d}_2}&\cdots\ar[r]^-{\widetilde{d}_n}&\widetilde{\mathscr{C}}^n\ar[r]^-{\widetilde{d}_{n+1}}&\widetilde{\mathscr{C}}^{n+1}\ar[r]&0,}
\end{equation}
where $\widetilde{d}_k$ is defined by, for $\left(\widetilde{\alpha},\widetilde{\eta}\right)\in\widetilde{\mathscr{C}}^k=\Omega^{k,0}\bigl(\textrm{End}\left(E\right)\bigr)\oplus\Omega^{k-1}\left(E\right)$,
\[\widetilde{d}_k\begin{pmatrix}
  \widetilde{\alpha}\\
  \widetilde{\eta}
\end{pmatrix}:=\begin{pmatrix}
  \theta^{\textrm{End}\left(E\right)}\wedge&0\\
  -ev\left(s\right)&-\theta\wedge
\end{pmatrix}\begin{pmatrix}
  \widetilde{\alpha}\\
  \widetilde{\eta}
\end{pmatrix}.\]
\subsection{Local moduli of solutions to the vortex equations for Higgs bundles}
Let $h$ be a Hermitian metric on $E$. For a real number $\tau$, we define a set $\mathscr{V}\left(E,h\right)_\tau$ by
\begin{equation}\label{eq:solution space}
  \mathscr{V}\left(E,h\right)_\tau:=\left\{\left(D,s\right)\middle|\left(D^{\prime\prime},s\right)\in\mathscr{HP}\left(E\right),\sqrt{-1}\Lambda R\left(D\right)+\frac{1}{2}s\circ s^*=\frac{\tau}{2}\textrm{id}_E\right\}
\end{equation}
this space can be regarded as the solution space of the $\tau$-vortex equation for Higgs bundle. The gauge group $\mathcal{G}\left(E,h\right)$ of $\left(E,h\right)$ acts on $\mathscr{V}\left(E,h\right)_\tau$ by, for $\left(D,s\right)\in\mathscr{V}\left(E,h\right)_\tau$ and $g\in\mathcal{G}\left(E,h\right)$,
\[\left(D,s\right)\cdot g:=\left(g^{-1}\circ D\circ g,g^{-1}s\right).\]
$\mathscr{M}_\tau\left(E,h\right)$ denotes the quotient space $\mathscr{V}\left(E,h\right)/\mathcal{G}\left(E,h\right)$ and is called the moduli space of solutions to the $\tau$-vortex equations for Higgs bundles.

For $k\in\mathbb{Z}_{\geq 0}$, $\mathscr{B}^k$ denotes the vector space $\Omega^k\bigl(\textrm{Herm}_{\textrm{skew}}\left(E,h\right)\bigr)$. A solution $\left(D,s\right)\in\mathscr{V}\left(E,h\right)$ induces the following elliptic complex.
\begin{equation}
  \left(\mathscr{B}^*\right):\xymatrix{0\ar[r]&\mathscr{B}^0\ar[r]^-{D_0}&\mathscr{C}^1\ar[r]^-{D_1}&\mathscr{B}^0\oplus\mathscr{C}^2\ar[r]^-{D_2}&\mathscr{C}^3\ar[r]^{D_3}&\cdots\ar[r]^-{D_{2n}}&\mathscr{C}^{2n}\ar[r]&0}
\end{equation}
where $D_k$ are defined by
\begin{align}
  &D_0\left(g\right):=\begin{pmatrix}
    D_{\textrm{End}\left(E\right)}g\\
    -g\left(s\right)
  \end{pmatrix},\,\textrm{for}\,\,g\in\mathscr{B}^0,\\
  &D_1\begin{pmatrix}
    \alpha\\
    \eta
  \end{pmatrix}:=\begin{pmatrix}
    \Lambda\left(\nabla^{\textrm{End}\left(E\right)}\beta+\Theta^{\textrm{End}\left(E\right)}\wedge\gamma\right)-\sqrt{-1}\left(s\circ\eta^*+\eta\circ s^*\right)\\
    D_{\textrm{End}\left(E\right)}^{\prime\prime}\left(\beta^{\prime\prime}+\gamma^\prime\right)\\
    -\left(\beta^{\prime\prime}+\gamma^\prime\right)s-D^{\prime\prime}\eta
  \end{pmatrix},\\
  &\quad\textrm{for}\,\,\left(\alpha,\eta\right)=\left(\beta+\gamma,\eta\right)\in\mathscr{C}^1=\Omega^1\bigl(\textrm{Herm}_{\textrm{skew}}\left(E,h\right)\bigr)\oplus\Omega^1\bigl(\textrm{Herm}\left(E,h\right)\bigr)\oplus\Omega^0\left(E\right),\\
  &D_2\begin{pmatrix}
    g\\
    \alpha\\
    \eta
  \end{pmatrix}:=d_2\begin{pmatrix}
    \alpha\\
    \eta
  \end{pmatrix},\,\textrm{for}\,\,\left(g,\alpha,\eta\right)\in\mathscr{B}^0\oplus\mathscr{C}^2,\\
  &D_k:=d_k\,\textrm{for}\,\,k\geq 3.
\end{align}
A complex $\left(\mathscr{B}^*\right)$ is also elliptic and this can be shown by a straightforward calculation. For a solution $\left(D,s\right)$, $H_{\left(D,s\right)}^k\left(\mathscr{B}^*\right)$ denotes the $k$-th cohomology group of the complex $\left(\mathscr{B}^*\right)$ induced by $\left(D,s\right)$.

We equip $\mathscr{C}^k, \mathscr{B}^0$ and $\mathscr{B}^0\oplus\mathscr{C}^2$ with the $L^2$-real inner products induced by a Hermitian metric $h$ and a K\"{a}hler metric $g$. Using these, we identify the cohomology groups $H^k\left(\mathscr{C}^*\right)$ and $H^k\left(\mathscr{B}^*\right)$ with the spaces of harmonic forms.

Next, for $\left(D,s\right)\in\mathscr{V}\left(E,h\right)$, we discuss the relationship between the cohomology groups $H_{\left(D^{\prime\prime},s\right)}^k\left(\mathscr{C}^*\right)$ and $H_{\left(D,s\right)}^k\left(\mathscr{B}^*\right)$ for $k\in\left\{0,2\right\}$. Before proceeding, we note that the following lemma holds.
\begin{lemma} \label{D_1^* representation}
  \textit{For $\left(f,\alpha,\eta\right)\in B^0\oplus\mathscr{C}^2$, the following holds.
  \begin{align}
    D_1^*\begin{pmatrix}
      f\\
      \alpha\\
      \eta
    \end{pmatrix}=\begin{pmatrix}
      \displaystyle D_{\textrm{End}\left(E\right)}^*\left(f\omega\right)+\frac{1}{2}P^{-1}\left(D_{\textrm{End}\left(E\right)}^{\prime\prime}\right)^*\alpha-\frac{1}{2}P^{-1}\left(\eta\circ s^*\right)\\
      2\sqrt{-1}f\left(s\right)-\left(D_{\textrm{End}\left(E\right)}^{\prime\prime}\right)^*\eta
    \end{pmatrix}
  \end{align}
  where $P:\Omega^1\bigl(\textrm{End}\left(E\right)\bigr)\to\Omega^1\bigl(\textrm{End}\left(E\right)\bigr)$ is defined by, for $\alpha\in\Omega^1\bigl(\textrm{End}\left(E\right)\bigr)$,
  \[P\left(\alpha\right)=\beta^{\prime\prime}+\gamma^\prime.\]}
\end{lemma}
Since we can prove this by a direct computation, we omit the proof.
\begin{proposition}\label{prop:cohomology_relationship}\ 

  \begin{enumerate}[$\left(1\right)$, leftmargin=20pt]
    \item \textit{$H_{\left(D^{\prime\prime},s\right)}^0\left(\mathscr{B}^*\right)\otimes\mathbb{C}\simeq H_{\left(D,s\right)}^0\left(\mathscr{C}^*\right)$.}
    \item \textit{$H_{\left(D,s\right)}^2\left(\mathscr{B}^*\right)\simeq H_{\left(D,s\right)}^0\left(\mathscr{B}^*\right)\oplus H_{\left(D^{\prime\prime},s\right)}^2\left(\mathscr{C}^*\right)$.}
  \end{enumerate}
\end{proposition}
\begin{proof}\ 
  
  \begin{enumerate}[$\left(1\right)$, leftmargin=20pt]
    \item For $f_1,f_2\in H_{\left(D,s\right)}^0\left(\mathscr{B}^*\right)$, it is clear that $D_{\textrm{End}\left(E\right)}^{\prime\prime}f_1=D_{\textrm{End}\left(E\right)}^{\prime\prime}f_2=0$ holds. Thus we have $f_1+\sqrt{-1}f_2\in H_{\left(D^{\prime\prime},s\right)}^0\left(\mathscr{C}^*\right)$. Conversely, let $g\in H_{\left(D^{\prime\prime},s\right)}^0\left(\mathscr{C}^*\right)$ and decompose it as
    \[g=g_1+\sqrt{-1}g_2,\quad g_1,g_2\in \mathscr{B}^0.\]
    Since we have $D_{\textrm{End}\left(E\right)}^{\prime\prime}g=0$, $D_{\textrm{End}\left(E\right)}^\prime g^*=0$ holds. Thus we have
    \begin{align}
      \sqrt{-1}\Lambda D_{\textrm{End}\left(E\right)}^{\prime\prime}D_{\textrm{End}\left(E\right)}^\prime g=&\sqrt{-1}\Lambda R\left(D_{\textrm{End}\left(E\right)}\right)g\\
      =&\sqrt{-1}\Lambda R\left(D\right)\circ g-g\circ\sqrt{-1}\Lambda R\left(D\right)\\
      =&\frac{\tau}{2}\textrm{id}_E\circ g-\frac{1}{2}s\circ s^*\circ g-g\circ\frac{\tau}{2}\textrm{id}_E+\frac{1}{2}g\circ s\circ s^*\\
      =&-\frac{1}{2}s\circ s^*\circ g.
    \end{align}
    and similarly,
    \[\sqrt{-1}\Lambda D_{\textrm{End}\left(E\right)}^\prime D_{\textrm{End}\left(E\right)}^{\prime\prime}g^*=\frac{1}{2}g^*\circ s\circ s^*.\]
    Now, we have
    \[\frac{1}{2}g^*\circ s\circ s^*=-\left(\frac{1}{2}s\circ s^*\circ g\right)^*=-\left(\sqrt{-1}\Lambda R\left(D_{\textrm{End}\left(E\right)}\right)g\right)^*=-\frac{1}{2}g^*\circ s\circ s^*.\]
    Therefore $g^*\circ s\circ s^*=0$ holds. By composing both sides with $g$ on the left and taking the trace, we obtain $g^*\left(s\right)=0$. Hence $D_{\textrm{End}\left(E\right)}^\prime g=D_{\textrm{End}\left(E\right)}^{\prime\prime}g^*=0$ holds and consequently, $g_1,g_2\in H_{\left(D,s\right)}^0\left(\mathscr{B}^*\right)$ holds.
    \item By Lemma \ref{D_1^* representation}, for $\left(f,\alpha,\eta\right)\in H_{\left(D,s\right)}^2\left(\mathscr{B}^*\right)$, we have
    \begin{align}
      D_1^*\left(f,\alpha,\eta\right)=0\overset{\textrm{iff}}{\Longleftrightarrow}&\left\{\begin{aligned}
        &P\bigl(D^*\left(f\omega\right)\bigr)+\frac{1}{2}\left(D_{\textrm{End}\left(E\right)}^{\prime\prime}\right)^*\alpha-\frac{1}{2}\eta\circ s^*=0,\\
        &2\sqrt{-1}f\left(s\right)-\left(D^{\prime\prime}\right)^*\eta=0
      \end{aligned}\right.\\
      \overset{\textrm{iff}}{\Longleftrightarrow}&\left\{\begin{aligned} \label{paraphrase}
        &-\sqrt{-1}D_{\textrm{End}\left(E\right)}^{\prime\prime}f+\frac{1}{2}\left(D_{\textrm{End}\left(E\right)}^{\prime\prime}\right)^*\alpha-\frac{1}{2}\eta\circ s^*=0,\\
        &2\sqrt{-1}f\left(s\right)-\left(D^{\prime\prime}\right)^*\eta=0.
      \end{aligned}\right.
    \end{align}
    Therfore the square of the $L^2$-norm of $f\left(s\right)\in\Omega^0\left(E\right)$ is given by the following.
    \begin{align}
      \langle f\left(s\right),f\left(s\right)\rangle_{L^2\left(M\right)}=&\langle\frac{1}{2\sqrt{-1}}\left(D^{\prime\prime}\right)^*\eta,f\left(s\right)\rangle_{L^2\left(M\right)}\\
      =&\langle\frac{1}{2\sqrt{-1}}\eta,\left(D_{\textrm{End}\left(E\right)}^{\prime\prime}f\right)\circ s\rangle_{L^2\left(M\right)}\\
      =&\langle\frac{1}{2\sqrt{-1}}\eta\circ s^*,D_{\textrm{End}\left(E\right)}^{\prime\prime}f\rangle_{L^2\left(M\right)}\\
      =&-\langle D_{\textrm{End}\left(E\right)}^{\prime\prime}f,D_{\textrm{End}\left(E\right)}^{\prime\prime}f\rangle_{L^2\left(M\right)}.
    \end{align}
    Hence we obtain $f\left(s\right)=0$ and $D_{\textrm{End}\left(E\right)}^{\prime\prime}f=0$ and consequently $f\in H_{\left(D,s\right)}^0\left(\mathscr{B}^*\right)$ and $\left(\alpha,\eta\right)\in H_{\left(D^{\prime\prime},s\right)}^2\left(\mathscr{C}^*\right)$ holds. This completes the proof.
  \end{enumerate}
\end{proof}
\subsection{The relationship between the moduli space of stable Higgs pairs and the moduli space of solutions to the vortex equations for Higgs bundles}
$\textrm{Herm}^+\left(E\right)$ denotes the space of Hermitian metrics on $E$. The gauge group $\mathcal{G}\left(E\right)$ of $E$ acts on $\textrm{Herm}^+\left(E\right)$ from the right by, for $h\in\textrm{Herm}^+\left(E\right)$ and $g\in\mathcal{G}\left(E\right)$,
\[\left(h\cdot g\right)\left(\cdot,\cdot\right):=h\bigl(g\left(\cdot\right),g\left(\cdot\right)\bigr).\]

For a structure of Higgs bundle $D^{\prime\prime}$ in E and a Hermitian metric $h$ on $E$, $D_{D^{\prime\prime},h}$ denotes the HS connection of a Hermitian Higgs bundle $\left(E,D^{\prime\prime},h\right)$ induced by $D^{\prime\prime}$ and $h$. When we fix a Hermitian metric $h$, $\mathcal{G}\left(E\right)$ also acts on the space of HS connections $\mathscr{H}\left(E,h\right)$ from the right by, for $D_{D^{\prime\prime},h}\in\mathscr{H}\left(E,h\right)$ and $g\in\mathcal{G}\left(E\right)$,
\[D_{D^{\prime\prime},h}\cdot g:=g^*\circ D^\prime\circ \left(g^*\right)^{-1}+g^{-1}\circ D^{\prime\prime}\circ g.\]
By a straightforward computation, we obtain the following proposition.
\begin{proposition}
  \textit{Fix a Hermitian metric $h$ on $E$. For a connection $D_{D^{\prime\prime},h}\in\mathscr{H}\left(E,h\right)$, a section $s\in\Omega^0\left(E\right)$ satisfying $D^{\prime\prime}s=0$ and gauge transformation $g\in\mathcal{G}\left(E\right)$, the following are equivalent.
  \begin{enumerate}
    \item $\left(D_{D^{\prime\prime},h}\cdot g,s\cdot g\right)\in\mathscr{V}_\tau\left(E,h\right)$.
    \item $\left(D_{D^{\prime\prime},h},s\right)\in\mathscr{V}_\tau\left(E,h\cdot g^{-1}\right)$.
  \end{enumerate}}
\end{proposition}
Fix a Hermitian metric $h$ on $E$. The space of solutions to the $\tau$-vortex equations for Higgs bundles $\mathscr{V}_\tau\left(E,h\right)$ is given by \eqref{eq:solution space}. By the Kobayashi-Hitchin correspondence, the space of $\tau$-stable Higgs pairs can be identified with the following space.
\begin{equation}
  \mathscr{HP}^{\textrm{st}}\left(E\right)=\left\{\left(D^{\prime\prime},s\right)\in\mathscr{HP}\left(E\right)\middle|\begin{gathered}
    \sqrt{-1}\Lambda R\left(D_{D^{\prime\prime},H}\right)+\frac{1}{2}s\circ s_H^*=\frac{\tau}{2}\textrm{id}_E \label{eq:vortex}\\
    \textrm{has a unique solution}\,H\in\textrm{Herm}^+\left(E\right).
  \end{gathered}\right\}
\end{equation}
Then $\mathscr{M}_{\textrm{HP}}^{\textrm{st}}\left(E\right)$ denotes the quotient space $\mathscr{HP}^{\textrm{st}}\left(E\right)/\mathcal{G}\left(E\right)$ and is called the moduli space of stable Higgs pairs. It is an open subset of moduli space $\mathscr{M}_{\textrm{HP}}\left(E\right)$ of Higgs pairs. From the uniqueness of the solution $H$ to equation \eqref{eq:vortex} for the $\tau$-stable pair, the following proposition holds.
\begin{proposition} \label{prop:one-to-one}
  \textit{There is a one-to-one correspondence between the moduli spaces $\mathscr{M}_\tau\left(E,h\right)$ and $\mathscr{M}_{\textrm{HP}}^{\textrm{st}}\left(E\right)$.}
\end{proposition}
\begin{definition}
  A Higgs pair $\left(D^{\prime\prime},s\right)\in\mathscr{HP}\left(E\right)$ is called simple if the $0$-th cohomology group $H_{\left(D^{\prime\prime},s\right)}^0\left(\mathscr{C}^*\right)$ vanishes.
\end{definition}
It is easy to prove that a $\tau$-stable pair $\left(D^{\prime\prime},s\right)$ is simple. By using a standard method \cite{Griffiths}, $\mathscr{M}_{\textrm{HP}}\left(E\right)$ is a non-singular complex manifold in neighbourhoods of points $\left[D^{\prime\prime},s\right]$ with $H_{\left(D^{\prime\prime},s\right)}^0\left(\mathscr{C}^*\right)=0$ and $H_{\left(D^{\prime\prime},s\right)}^2\left(\mathscr{C}^*\right)=0$ hold and $\mathscr{M}_\tau\left(E,h\right)$ is also a non-singular complex manifold in neighborhoods of points $\left[D,s\right]$ with $H_{\left(D,s\right)}^0\left(\mathscr{B}^*\right)=0$ and $H_{\left(D,s\right)}^2\left(\mathscr{B}^*\right)=0$ hold. Combining Proposition \ref{prop:cohomology_relationship} and \ref{prop:one-to-one}, the following proposition holds.
\begin{proposition}
  \textit{The moduli space $\mathscr{M}_{\textrm{HP}}^{\textrm{st}}\left(E\right)$ of $\tau$-stable Higgs pairs is a non-singular complex manifold in neighbourhoods of points $\left[D^{\prime\prime},s\right]$ with $H_{\left(D^{\prime\prime},s\right)}^2\left(\mathscr{C}^*\right)=0$ holds and the moduli space $\mathscr{M}_\tau\left(E,h\right)$ of solutions to the $\tau$-vortex equations for Higgs bundles is also a non-singular complex manifold in neighbourhoods of points $\left[D_{D^{\prime\prime},h},s\right]$ with $H_{\left(D^{\prime\prime},s\right)}^2\left(\mathscr{C}^*\right)=0$ holds. Moreover, these spaces are biholomorphic.}
\end{proposition}
\section{Properties Of Moduli Spaces When The Base Space Is A Riemann Surface}
Let $M$ be a compact Riemann surface of genus $g\geq 2$ and $E$ be a smooth complex vector bundle over $M$. Fix a K\"{a}hler metric $g_M$ on $M$. In this section, we discuss the properties of the moduli space of $\tau$-stable Higgs pairs when the base space is one-dimensional manifold. We define a real number $\mu_+$ as the smallest possible value after $\mu\left(E\right)$ for the slope of Higgs subbundle of $E$. Throughout this section, we impose the following condition on $\deg E,\rank E$ and $\tau$.
\begin{assumption}\label{eq:tau_assumption}
  $\deg E$ and $\rank E$ are coprime and $\tau$ satisfies the inequality
  \[\mu\left(E\right)<\frac{\textrm{Vol}\left(M,g_M\right)}{4\pi}\tau<\mu_+.\]
\end{assumption}
Under Assumption \ref{eq:tau_assumption}, the following lemma holds.
\begin{proposition}\label{prop:semistable}
  \textit{Under an assumption \eqref{eq:tau_assumption}, for a $\tau$-stable Higgs pair $\left(D^{\prime\prime},s\right)$, $\left(E,D^{\prime\prime}\right)$ is a stable Higgs bundle.}
\end{proposition}
Since Proposition \ref{prop:semistable} is analogous to \cite[Proposition $1.7$]{Bradlow3}, we omit the proof.

Firstly, we prove that the moduli space $\mathscr{M}_{\textrm{HP}}^\textrm{st}\left(E\right)$ is globally non-singular. A point $\left[D^{\prime\prime},s\right]\in\mathscr{M}_{\textrm{HP}}^{\textrm{st}}\left(E\right)$ is a smooth point when the $2$-nd hypercohomology group $\mathbb{H}^2$ of $\widetilde{\mathscr{C}}^*$ vanishes. Since $\dim_{\mathbb{C}}M=1$ holds, the complex $\left(\widetilde{\mathscr{C}}^*\right)$ is given by
\begin{equation}
  \left(\widetilde{\mathscr{C}}^*\right):\xymatrix{0\ar[r]&\textrm{End}\left(E\right)\ar[r]^-{\widetilde{d}_0}&\textrm{End}\left(E\right)\otimes K_M\oplus E\ar[r]^-{\widetilde{d}_1}&E\otimes K_M\ar[r]&0},
\end{equation}
where $K_M$ denotes the canonical line bundle of $M$. There is a spectral sequence $\left\{E^{p,q}_r\right\}$ with $E^{p,q}_r\Rightarrow\mathbb{H}^{p+q}$ holds. In the present case, we have $\mathbb{H}^2=E_3^{2,0}\oplus E_3^{1,1}$ and $E_3^{1,1}=E_2^{1,1}$. $E_3^{2,0}$ is given by
\[E_3^{2,0}=\frac{E_2^{2,0}}{\Im{E_2^{0,1}\to E^{2,0}}}\]
and $E_2^{2,0}$ and $E_2^{1,1}$ are given by
\[E_2^{2,0}=\frac{H^0\left(E\otimes K_M\right)}{\Im{\widetilde{d}_{1*}:H^0\left(\textrm{End}\left(E\right)\otimes K_M\right)\oplus H^0\left(E\right)\to H^0\left(E\otimes K_M\right)}},\]
\[E_2^{1,1}=\frac{\textrm{Ker}\left\{\widetilde{d}_{1*}:H^1\left(\textrm{End}\left(E\right)\otimes K_M\right)\oplus H^1\left(E\right)\to H^1\left(E\otimes K_M\right)\right\}}{\Im{\widetilde{d}_{0*}:H^1\bigl(\textrm{End}\left(E\right)\bigr)\to H^1\left(\textrm{End}\left(E\right)\otimes K_M\right)\oplus H^1\left(E\right)}}.\]
Since $\left(D^{\prime\prime},s\right)$ is $\tau$-stable, by the Kobayashi-Hitchin correspondence, there exists a Hermitian metric $h$ on $E$ such that $\left(D^{\prime\prime},s\right)$ is a solution to the $\tau$-vortex equation. By using this metric, we prove the following propositions.
\begin{proposition}\label{prop:H^2_vanishing_1}
  \textit{$E_2^{1,1}=0$ holds.}
\end{proposition}
\begin{proof}
  We identify $H^1\bigl(\textrm{End}\left(E\right)\otimes K_M\bigr)$ (resp. $H^1\left(E\right)$) with the space of harmonic $\textrm{End}\left(E\right)$-valued $\left(1,1\right)$ (resp. $E$-valued $\left(0,1\right)$)-forms. Under this identification, $E_2^{1,1}$ is given by
  \[E_2^{1,1}=\frac{\left\{\left(f\omega_M,\alpha^{\prime\prime}\right)\in H^1\bigl(\textrm{End}\left(E\right)\bigr)\oplus H^1\left(E\right)\middle|\begin{gathered}
      \partial^{\textrm{End}\left(E\right)}f=0,\\
      H\left(f\left(s\right)\omega_M+\theta\wedge\alpha^{\prime\prime}\right)=0
    \end{gathered}\right\}}{\left\{\Bigl(H\left(\left[\theta\wedge\beta^{\prime\prime}\right]\right),-H\bigl(\beta^{\prime\prime}\left(s\right)\bigr)\Bigr)\middle|\beta^{\prime\prime}\in H^1\bigl(\textrm{End}\left(E\right)\bigr)\right\}},\]
  where $H$ denotes the harmonic projection. Since $h$ and $g_M$ induces an inner product on $H^1\bigl(\textrm{End}\left(E\right)\bigr)\oplus H^1\left(E\right)$, $E_2^{1,1}$ is also given by
  \[E_2^{1,1}=\left\{\left(f\omega_M,\alpha^{\prime\prime}\right)\in H^1\bigl(\textrm{End}\left(E\right)\bigr)\oplus H^1\left(E\right)\middle|\begin{gathered}
  \partial^{\textrm{End}\left(E\right)}f=0,\left[\theta^*\wedge f\right]=0,\\
  \alpha^{\prime\prime}\otimes s^*=0,\\
  H\left(f\left(s\right)\omega_M+\theta\wedge\alpha^{\prime\prime}\right)=0
  \end{gathered}\right\}.\]
  Since $s$ is not a zero-section and $\alpha^{\prime\prime}$ is harmonic, we have $\alpha^{\prime\prime}=0$. Moreover, since $D^{\prime\prime}f^*=0$ holds and $\left(E,D^{\prime\prime}\right)$ is stable, there exists a $z\in\mathbb{C}$ such that $f=z\,\textrm{id}_E$ holds. Now, since $\left(z\,\textrm{id}_E,0\right)\in E_2^{1,1}\subset\mathbb{H}^2$ satisfies $d_2\left(z\,\textrm{id}_E\omega_M,0\right)=0$, we have $zs\omega_M=0$. Thus $z=0$ holds and this completes the proof.
\end{proof}
\begin{proposition}\label{prop:H^2_vanishing_2}
  \textit{$E_3^{2,0}=0$ holds.}
\end{proposition}
\begin{proof}
  We identify $E_3^{2,0}$ with the subspace of $H^{1,0}\left(E\right)$. For $\alpha^\prime\in E_3^{2,0}$, since we have $\left(0,\alpha^{\prime\prime}\right)\in H^2, d_1^*\left(0,\alpha^\prime\right)=0$ holds. Thus $\alpha^\prime\otimes s^*=0$ holds and this implies $\alpha^\prime=0$. This completes the proof.
\end{proof}
As a corollary, we obtain the following result.
\begin{corollary}
  \textit{Under Assumption \ref{eq:tau_assumption}, the moduli space $\mathscr{M}_{\textrm{HP}}^{\textrm{st}}\left(E\right)$ of $\tau$-stable Higgs pairs is a non-singular complex manifold.}
\end{corollary}
\begin{remark}
  Mehta established only the smoothness of a certain open subset of the moduli space of $\tau$-stable Higgs pairs \cite{Mehta}. Unlike his work, we establish the smoothness of the entire moduli space and our approach is different from that of his. 
\end{remark}
Secondly, we compute the Betti numbers of the moduli space $\mathscr{M}_{\textrm{HP}}^{\textrm{st}}\left(E\right)$ under an assumption that $\rank E=2$ and $\mu\left(E\right)>2g-2$. We use a method which is essentially the same as that used by Hitchin \cite{Hitchin} in his calculation of Betti numbers of the moduli space of stable Higgs bundles. Biswas and Schumacher \cite{Biswas} constructed a K\"{a}hler metric on the moduli space of quiver bundles. Thus the moduli space $\mathscr{M}_{\textrm{HP}}^{\textrm{st}}\left(E\right)$ has a K\"{a}hler metric. Let $h$ be a Hermitian metric on $E$. The moduli space $\mathscr{M}_{\textrm{HP}}^{\textrm{st}}\left(E\right)$ admits a Hamiltonian $S^1$-action defined by, for $\left[\bar{\partial}^E,\theta,s\right]\in\mathscr{M}_{\textrm{HP}}^{\textrm{st}}\left(E\right)$ and $z\in S^1$,
\[\left[\bar{\partial}^E,\theta,s\right]\cdot z:=\left[\bar{\partial}^E,z\theta,s\right].\]
The moment map $\mu:\mathscr{M}_{\textrm{HP}}^{\textrm{st}}\left(E\right)\to\mathbb{R}$ of this action is given by
\begin{equation}
  \mu\left(\left[\bar{\partial}^E,\theta,s\right]\right)=\left\|\theta\right\|_{L^2\left(M\right)}^2. \label{eq:moment map}
\end{equation}
By Assumption \ref{eq:tau_exclusion} and Lemma \ref{lem:L^4-identity} $\left(1\right)$, the map $\mu$ is proper.

The critical set of $\mu$ coincides with the fixed point set of the $S^1$-action. Obviously, $\mu^{-1}\left(0\right)$ is contained in the fixed point set. This subset coincides with the moduli space of $\tau$-stable holomorphic pairs, denoted by $N_0$. We consider the points fixed by the $S^1$-action with $\theta\ne 0$. Let $\left[\bar{\partial}^E,\theta,s\right]\in\mathscr{M}_{\textrm{HP}}^{\textrm{st}}\left(E\right)$ be a fixed point of the $S^1$-action. Then there exists gauge transformations $g\left(\varphi\right)\in\mathcal{G}\left(E\right)$ such that
\begin{empheq}[left=\empheqlbrace]{align}
  &g\left(\varphi\right)^{-1}\circ\bar{\partial}^E\circ g\left(\varphi\right)=\bar{\partial}^E, \label{eq:reducible}\\
  &g\left(\varphi\right)^{-1}\circ\theta\circ g\left(\varphi\right)=\exp(\sqrt{-1}\varphi)\theta,\\
  &g\left(\varphi\right)^{-1}s=s
\end{empheq}
hold. Equation \eqref{eq:reducible} implies that there exists a holomorphic subbundle $L$ of $\left(E,\bar{\partial}^E\right)$ such that $E=L\oplus\left(L^*\otimes\det E\right)$ holds. This vector bundle $L$ is an eigenbundle of the $S^1$-action on $E$ and we set $n\in\mathbb{Z}$ to be the weight of $L$. Then for $v\in L, \theta\left(v\right)$ is an eigenvector of $g\left(\varphi\right)$ corresponding to an eigenvalue $\exp(\sqrt{-1}\left(n-1\right)\varphi)$. Thus we choose a subbundle $L$ so that
\[\theta=\begin{pmatrix}
  0&0\\
  \psi&0
\end{pmatrix}\]
holds, where $\psi\in H^0\left(L^{-2}\otimes\det E\otimes K_M\right)$.

Since $s\in\Omega^0\left(E\right)$ is a fixed point of the $S^1$-action and the weights of $L$ and $L^{-1}\otimes\det E$ are distinct, by lemma \ref{lemma:s-nonzero}, either $s\in\Omega^0\left(L\right)$ or $s\in\Omega^0\left(L^{-1}\otimes\det E\right)$ holds. If $s$ is a section of $L$, $\psi s=0$ holds. However, since $s$ forms the frame of line bundle $L$ outside of zero set of $s$, we have $\psi=0$, which contradicts the assumption that $\theta\ne 0$. Hence $s\in\Omega^0\left(L^{-1}\otimes\det E\right)$ holds.

Since $\psi$ and $s$ are holomorphic and are not a zero-section,
\[\deg L\leq\min\left\{\deg E,g-1+\frac{1}{2}\deg E\right\}=:m\]
holds. Moreover since $\left(\bar{\partial}^E,\theta,s\right)$ is $\tau$-stable and $L^{-1}\otimes\det E$ is Higgs subbundle of $\left(E,D^{\prime\prime}\right)$,
\[\deg E-\frac{\textrm{Vol}\left(M,g_M\right)}{4\pi}\tau<\frac{\textrm{Vol}\left(M,g_M\right)}{4\pi}\tau<\deg L\]
holds. Therefore we have an inequality
\[\frac{\textrm{Vol}\left(M,g_M\right)}{4\pi}\tau<\deg L\leq m.\]
Hence for integers $d$ with $\displaystyle\lfloor\frac{\textrm{Vol}\left(M,g\right)}{4\pi}\tau\rfloor+1\leq d\leq\lfloor m\rfloor$, where $\lfloor x\rfloor$ denotes the integer part of $x\in\mathbb{R}$, we define sets $N_d$ as
\[N_d:=\left\{\left[E=L\oplus\left(L^{-1}\otimes\det E\right),\begin{pmatrix}
0&0\\
\psi&0
\end{pmatrix},\begin{pmatrix}
0\\
s
\end{pmatrix}\right]\middle|\deg L=d\right\}.\]
Then the union $N_d$ forms the fixed point set of the $S^1$-action with $\theta\ne 0$. For $N_d$, the following proposition holds.
\begin{proposition}
  \textit{$N_d$ is biholomorphic to $\textrm{Sym}^{-2d+k+2g-2}M\times\textrm{Sym}^{k-d}M$, where $k=\deg E$.}
\end{proposition}
\begin{proof}
  For effective divisors $D\in\textrm{Sym}^{-2d+k+2g-2}M$ and $D^\prime\in\textrm{Sym}^{k-d}M$, we define a holomorphic map $f:\textrm{Sym}^{-2d+k+2g-2}M\times\textrm{Sym}^{k-d}M\to\textrm{Pic}^k\left(M\right)$ by
  \begin{equation}
    f\left(D,D^\prime\right):=\left[\mathcal{O}_M\left(D^\prime\right)^2\otimes\mathcal{O}_M\left(D\right)^{-1}\otimes K_M\right]. \label{eq:f}
  \end{equation}
  Given an element of $N_d$, we obtain an element of the graph $\Gamma\left(f\right)\subset\textrm{Sym}^{-2d+k+2g-2}M\times\textrm{Sym}^{k-d}M\times\textrm{Pic}^k\left(M\right)$ of $f$ so we can define a map from $N_d$ to $\textrm{Sym}^{-2d+k+2g-2}M\times\textrm{Sym}^{k-d}M$. To prove the claim, it is sufficient to show that $f$ is bijective. Since $M$ is a compact Riemannian surface, we can identify $\textrm{Pic}^k\left(M\right)$ as the moduli space of holomorphic structures in $\det E$. Let $\left(D,D^\prime,\det E\right)\in\Gamma\left(f\right)$, where $\det E$ denotes a smooth determinant line bundle of $E$ with some holomorphic structure. From this element, we obtain a holomorphic line bundle $L:=\mathcal{O}_M\left(D^\prime\right)^{-1}\otimes\det E$, a holomorphic section $\psi$ of $\mathcal{O}_M\left(D\right)=L^{-2}\otimes\det E\otimes K_M$ and $s$ of $\mathcal{O}_M\left(D^\prime\right)=L^{-1}\otimes\det E$. $\psi$ and $s$ are unique up to complex constant multiplicity. But for $w_1,w_2\in\mathbb{C}$, we can check easily that $\left(L\oplus\left(L^{-1}\otimes\det E\right),\psi,s\right)$ and $\left(L\oplus\left(L^{-1}\otimes\det E\right),w_1\psi,w_2s\right)$ are in the same orbit. This completes the proof.
\end{proof}
The section $s$ is fixed under the $S^1$-action on the vector bundle $E$. Since $s$ is a section of $L^{-1}\otimes\det E$, it follows that the weight of $L$ is $-1$ and that of $L^{-1}\otimes\det E$ is $0$. The index of $N_d$ is given by the real dimension of the $1$-st cohomology group of the following elliptic complex, consisting of the components of $\left(\mathscr{C}^*\right)$ with positive weight with respect to the $S^1$-action.
\[\xymatrix{\Omega^0\bigl(\textrm{Hom}\left(L,L^{-1}\otimes\det E\right)\bigr)\ar[r]^-{\bar{\partial}}&\Omega^{0,1}\bigl(\textrm{Hom}\left(L,L^{-1}\otimes\det E\right)\bigr).}\]
Thus the index of $N_d$ is given by $2\left(2d+g-k-1\right)$ and we obtain the following theorem.
\begin{theorem}\label{thm:Poincare}
  \textit{The Poincar\'{e} polynomial $P\left(t\right)$ of the moduli space of $\tau$-stable Higgs pairs $\mathscr{M}_{\textrm{HP}}^{\textrm{st}}\left(E\right)$ is given by the coefficient of $x^{\deg E+2g}y^{\deg E+2g}$ in
  \begin{align}
    x^{2g+1+\lfloor\tau\rfloor}y^{\deg E+2g}\frac{\left(1+t\right)^{2g}\left(1+tx\right)^g\left(1+tx\right)^g}{\left(1-t^2\right)\left(1-x\right)\left(1-t^2x\right)}\left(\frac{t^{2\left(\deg E-1-\lfloor\tau\rfloor\right)}}{1-t^{-2}x}-\frac{t^{2\left(g+1-\deg E+2\lfloor\tau\rfloor\right)}}{1-t^4x}\right)\\
    +\sum_{d=\lfloor\frac{\textrm{Vol}\left(M,g\right)}{4\pi}\tau\rfloor+1}^mt^{2\left(2d+g-\deg E-1\right)}x^{2d+2}y^{d-2g}\frac{\left(1+tx\right)^{2g}\left(1+ty\right)^{2g}}{\left(1-x\right)\left(1-y\right)\left(1-t^2x\right)\left(1-t^2y\right)}
  \end{align}}
\end{theorem}
\begin{proof}
  In \cite{Munoz}, the Poincar\'{e} polynomial of the moduli space of $\tau$-stable holomorphic pairs is computed and is equal to the following.
  \[\coeff_{x^0}\left[\frac{\left(1+t\right)^{2g}\left(1+tx\right)^g\left(1+tx\right)^g}{\left(1-t^2\right)\left(1-x\right)\left(1-t^2x\right)x^{\deg E-1-\lfloor\tau\rfloor}}\left(\frac{t^{2\left(\deg E-1-\lfloor\tau\rfloor\right)}}{1-t^{-2}x}-\frac{t^{2\left(g+1-\deg E+2\lfloor\tau\rfloor\right)}}{1-t^4x}\right)\right].\]
  Also, in \cite{Macdonald}, it is shown that the the Poincar\'{e} polynomial of symmetric product of Riemann surface is given by
  \[P_t\left(\textrm{Sym}^nM\right)=\coeff_{x^n}\frac{\left(1+tx\right)^{2g}}{\left(1-x\right)\left(1-t^2x\right)}.\]
  The claim follows from the above discussion.
\end{proof}
Lastly, we discuss about a map from the moduli space of $\tau$-stable Higgs pairs to the moduli space of stable Higgs bundles. Since we suppose Assumption \ref{eq:tau_assumption}, we have a mapping $\pi^\prime:\mathscr{M}_{\textrm{HP}}^{\textrm{st}}\left(E\right)\to\mathscr{M}_{\textrm{Higgs}}^{\textrm{st}}\left(E\right)$ defined by, for $\left[\left(D^{\prime\prime},s\right)\right]\in\mathscr{M}_{\textrm{HP}}^{\textrm{st}}\left(E\right), \pi^\prime\left[\left(D^{\prime\prime},s\right)\right]=\left[D^{\prime\prime}\right]$, where $\mathscr{M}_{\textrm{Higgs}}^{\textrm{st}}\left(E\right)$ is the moduli space of stable Higgs bundles. This mapping can be coinsidered as a generalization of a mapping $\pi$ defined in \cite[p. $511$, $\left(6.1\right)$]{Bradlow3}. $\pi$ is a fibration over the moduli space of stable vector bundles with projective fiber $\mathbb{CP}^{h^0\left(E\right)-1}$. We have the following proposition which is analogous to $\pi$.
\begin{proposition}\label{prop:fibration}
  \textit{Consider an open subset $\mathscr{M}_0$ of $\mathscr{M}_{\textrm{Higgs}}^{\textrm{st}}\left(E\right)$ consisting of $\left[D^{\prime\prime}\right]$ such that $\alpha=0$ is the only solution to an equation $\Delta^{D^{\prime\prime},h}a=0$ in $\Omega^1\left(E\right)$. Then a mapping $\pi^\prime:\mathscr{M}_{\textrm{HP}}^{\textrm{st}}\left(E\right)\to\mathscr{M}_{\textrm{Higgs}}^{\textrm{st}}\left(E\right)$ is a fibration over $\mathscr{M}_0$.}
\end{proposition}
\begin{proof}
  Since the function $\mu:\mathscr{M}_{\textrm{HP}}^{\textrm{st}}\left(E\right)\to\mathbb{R}$ defined as in \eqref{eq:moment map} is proper, so is $\pi^\prime$. The derivative $d\pi^\prime$ of $\pi^\prime$ is given by the map $\mathbb{H}^1\left(C_2^*\right)\to\mathbb{H}^1\left(C_3^*\right)$ between the hypercohomology groups induced by the following short exact sequence of complexes.
  \[\xymatrix{0\ar[r]\ar[d]&\textrm{End}\left(E\right)\ar[rr]^-{\textrm{id}}\ar[d]&&\textrm{End}\left(E\right)\ar[d]\\
  E\ar[r]^-{\textrm{inclusion}}\ar[d]&\textrm{End}\left(E\right)\otimes K_M\oplus E\ar[rr]^-{\textrm{projection}}\ar[d]&&\textrm{End}\left(E\right)\otimes K_M\ar[d]\\
  E\otimes K_M\ar[r]^-{\textrm{id}}&E\otimes K_M\ar[rr]&&0,}\]
  where the left, middle and right complexes are denoted by $\left(C_1^*\right),\left(C_2^*\right),\left(C_3^*\right)$, respectively. $d\pi^\prime$ is surjective if and only if $\mathbb{H}^2\left(C_1^*\right)=0$ holds. Thus from the Ehresmann's theorem, the claim follows.
\end{proof}
Since for $\left[D^{\prime\prime}\right]\in\mathscr{M}_0$, the stabilizer of $D^{\prime\prime}$ in gauge group $\mathcal{G}\left(E\right)$ correspondes to the constant multiples of the identity, the fiber $\left(\pi^\prime\right)^{-1}\left(\left[D^{\prime\prime}\right]\right)$ can be identified with the projective space $\mathbb{P}\left(\mathbb{H}^1\left(C_1^*\right)\right)$.
\begin{remark}
  Bradlow and Daskalopoulos \cite{Bradlow3} assumed that $\mu\left(E\right)>2g-2$ when they proved that the moduli space of $\tau$-stable holomorphic pairs $N_0$ is smooth. They also used this assumption when they proved that the moduli space $N_0$ is a projective fibration over the moduli space of stable holomorphic structures. However, Garcia-Prada \cite{Garcia-Prada} proved that the space $N_0$ is smooth without this assumption. In this case, $\pi$ need not be a projective fibration. This means that the dimension of the fibers of $\pi$ may jump. In the present situation, the same behavior may occur for $\pi^\prime$, that is, the dimension of the fibers of $\pi^\prime$ may also jump. 
\end{remark}
\section*{Acknowledgement}
This work was supported by JSPS KAKENHI Grant Number JP25KJ1222.
\bibliography{reference}

@inproceedings{Hitchin,
  author       = {N. J. Hitchin},
  title        = {{T}he self-duality equations on a {R}iemann surface},
  booktitle    = {Proc. London Math. Soc.},
  volume       = {s3-55},
  number       = {1},
  pages        = {59-126},
  url          = {https://londmathsoc.onlinelibrary.wiley.com/doi/abs/10.1112/plms/s3-55.1.59},
  year         = {1987}
}

@inproceedings{Simpson,
  author       = {C. T. Simpson},
  title        = {{C}onstructing variations of {H}odge structure using {Y}ang-{M}ills theory and applications to uniformization},
  booktitle    = {J. Amer. Math. Soc.},
  volume       = {1},
  number       = {4},
  pages        = {867--918},
  URL          = {http://www.jstor.org/stable/1990994},
  year         = {1988}
}

@inproceedings{Simpson2,
  author       = {C. T. Simpson},
  title        = {{H}iggs bundles and local systems},
  booktitle    = {Publications Math\'ematiques de l'IH\'ES},
  volume       = {75},
  number       = {},
  pages        = {5--95},
  url          = {https://www.numdam.org/item/PMIHES_1992__75__5_0/},
  year         = {1992}
}

@inproceedings{Bradlow2,
  author       = {S. B. Bradlow},
  title        = {{S}pecial metrics and stability for holomorphic bundles with global sections},
  booktitle    = {Journal of Differential Geometry},
  volume       = {33},
  number       = {1},
  pages        = {169--213},
  url          = {https://doi.org/10.4310/jdg/1214446034},
  year         = {1991}
}

@inproceedings{Bradlow3,
  author       = {S. B. Bradlow and G. D. Daskalopoulos},
  title        = {{M}oduli of stable pairs for holomorphic bundles over {R}iemann surfaces},
  booktitle    = {International Journal of Mathematics},
  volume       = {02},
  number       = {05},
  pages        = {477--513},
  url          = {https://doi.org/10.1142/S0129167X91000272},
  year         = {1991},
}

@inproceedings{Bruzzo,
  author       = {U. Bruzzo and B. G. Otero},
  title        = {{M}etrics on semistable and numerically effective {H}iggs bundles},
  booktitle    = {{J}ournal für die reine und angewandte {M}athematik},
  volume       = {2007},
  number       = {612},
  pages        = {59--79},
  url          = {https://doi.org/10.1515/CRELLE.2007.084},
  year         = {2007}
}

@inproceedings{Uhlenbeckyau,
  author       = {K. Uhlenbeck and S. T. Yau},
  title        = {{O}n the existence of {H}ermitian-{Y}ang-{M}ills connections in stable vector bundles},
  booktitle    = {Communications on Pure and Applied Mathematics},
  volume       = {39},
  number       = {S1},
  pages        = {S257--S293},
  url          = {https://onlinelibrary.wiley.com/doi/abs/10.1002/cpa.3160390714},
  year         = {1986}
}

@inproceedings{Garcia-Prada,
  author       = {O. Garc\'{i}a-Prada},
  title        = {{D}imensional reduction of stable bundles, vortices and stable pairs},
  booktitle    = {International Journal of Mathematics},
  volume       = {05},
  number       = {01},
  pages        = {1--52},
  URL          = {https://doi.org/10.1142/S0129167X94000024},
  year         = {1994}
}

@inproceedings{Alvarez-Consul-Garcia-Prada,
  author       = {{L. \'{A}lvarez-C\'{o}nsul} and O. Garc\'{i}a-Prada},
  title        = {{H}itchin-{K}obayashi correspondence, quivers, and vortices},
  booktitle    = {Communications in Mathematical Physics},
  volume       = {238},
  number       = {},
  pages        = {1--33},
  URL          = {https://link.springer.com/article/10.1007/s00220-003-0853-1},
  year         = {2003}
}

@inproceedings{Biswas,
  author       = {I. Biswas and G. Schumacher},
  title        = {{D}ifferential geometry of moduli spaces of quiver bundles},
  booktitle    = {Journal of Geometry and Physics},
  volume       = {118},
  number       = {},
  pages        = {51--66},
  url          = {https://www.sciencedirect.com/science/article/pii/S0393044016302856},
  year         = {2017}
}

@inproceedings{Munoz,
  author       = {V. Muñoz and D. Ortega and M. J. V\'{a}zquez-Gallo},
  title        = {{H}odge polynomials of the moduli spaces of pairs},
  booktitle    = {International Journal of Mathematics},
  volume       = {18},
  number       = {06},
  pages        = {695--721},
  url          = {https://doi.org/10.1142/S0129167X07004266},
  year         = {2007}
}

@inproceedings{Macdonald,
  author       = {I.G. MacDonald},
  title        = {{S}ymmetric products of an algebraic curve},
  booktitle    = {Topology},
  volume       = {1},
  number       = {4},
  pages        = {319--343},
  url          = {https://www.sciencedirect.com/science/article/pii/0040938362900198},
  year         = {1962}
}

@inproceedings{Gothen-King,
  author       = {P. B. Gothen and A. D. King},
  title        = {{H}omological algebra of twisted quiver bundles},
  booktitle    = {Journal of the London Mathematical Society},
  volume       = {71},
  number       = {1},
  pages        = {85--99},
  url          = {https://londmathsoc.onlinelibrary.wiley.com/doi/abs/10.1112/S0024610704005952},
  year         = {2005}
}

@inproceedings{Mehta,
  author       = {M. Mehta},
  title        = {{B}irational equivalence of {H}iggs moduli},
  booktitle    = {International Journal of Mathematics},
  volume       = {16},
  number       = {04},
  pages        = {365--386},
  url          = {https://doi.org/10.1142/S0129167X05002916},
  year         = {2005}
}

@misc{Siqi,
  author       ={S. He},
  title        ={{T}he behavior of sequences of solutions to the {H}itchin-{S}impson equations},
  note         ={ar{X}iv:2002.08109},
  year         ={2024},
  url          ={https://arxiv.org/abs/2002.08109}, 
}

@misc{Ono,
  author       ={T. Ono},
  title        ={{D}imensional reduction of stable {H}iggs bundle and the doubly-coupled vortex equations},
  note         ={ar{X}iv:2509.07489},
  year         ={2025},
  url          ={https://arxiv.org/abs/2509.07489}, 
}

@book{Griffiths,
  author       = {P. A. Griffiths},
  publisher    = {Springer Berlin, Heidelberg},
  title        = {The Extension Problem for Compact Submanifolds of Complex Manifolds {I}},
  URL          = {https://link.springer.com/chapter/10.1007/978-3-642-48016-4_12},
  year         = {1965}
}
\bibliographystyle{amsplain}
\end{document}